\documentclass[a4paper,11pt]{article} 

\RequirePackage{fix-cm}

\usepackage{graphicx}
\usepackage{epstopdf}
\usepackage{probability}
\usepackage{enumerate}
\usepackage{algorithm}
\usepackage{algorithmic}
\usepackage{amsfonts}
\usepackage{amsthm}
\usepackage{bm}
\usepackage[mathscr]{eucal}
\usepackage{a4wide}
\usepackage{authblk}
\usepackage{tikz}
\usepackage[colorlinks=true,allcolors=blue]{hyperref}
\usepackage[square,comma,compress,numbers]{natbib}

\newcommand\numberthis{\addtocounter{equation}{1}\tag{\theequation}}

\allowdisplaybreaks

\newcommand{\G}{\mathcal{G}}
\newcommand{\R}{\mathbb{R}}
\newcommand{\D}{\mathscr{D}}
\renewcommand{\S}{\mathcal{S}}

\newtheorem{theorem}{Theorem}[section]

\newtheorem{lemma}[theorem]{Lemma}
\newtheorem{proposition}[theorem]{Proposition}

\let\tmp\alpha
\let\alpha\gamma
\let\gamma\tmp

\title{Sparse Maximum-Entropy Random Graphs\\ with a Given Power-Law Degree Distribution}
\author[1]{Pim van der Hoorn}
\author[2]{Gabor Lippner}
\author[1,2,3]{Dmitri Krioukov}
\affil[1]{Northeastern University, Department of Physics}
\affil[2]{Northeastern University, Department of Mathematics}
\affil[3]{Northeastern University, Departments of Electrical\&Computer Engineering}

\begin{document}

\maketitle

\begin{abstract}

Even though power-law or close-to-power-law degree distributions are ubiquitously observed in a great variety of large real networks, the mathematically satisfactory treatment of random power-law graphs satisfying basic statistical requirements of realism is still lacking. These requirements are: sparsity, exchangeability, projectivity, and unbiasedness. The last requirement states that entropy of the graph ensemble must be maximized under the degree distribution constraints. Here we prove that the hypersoft configuration model (HSCM), belonging to the class of random graphs with latent hyperparameters, also known as inhomogeneous random graphs or $W$-random graphs, is an ensemble of random power-law graphs that are sparse, unbiased, and either exchangeable or projective. The proof of their unbiasedness relies on generalized graphons, and on mapping the problem of maximization of the normalized Gibbs entropy of a random graph ensemble, to the graphon entropy maximization problem, showing that the two entropies converge to each other in the large-graph limit.
\end{abstract}

\newpage

\tableofcontents

\newpage

\section{Introduction}

Random graphs have been used extensively to model a variety of real networks. Many of these networks, ranging from the Internet and social networks to the brain and the universe, have broad degree distributions, often following closely power laws~\cite{boccaletti2006complex,newman2010networks,barabasi2016network}, that the simplest random graph model, the Erd\H{o}s-R\'enyi random graphs~\cite{solomonoff1951connectivity,gilbert1959random,erdos1959random} with Poisson degree distributions, does not reproduce. To resolve this disconnect, several alternative models have been proposed and studied. The first one is the configuration model (CM), random graphs with a given degree sequence~\cite{bender1978asymptotic,molloy1995critical}. This model is a microcanonical ensemble of random graphs. Every graph in the ensemble has the same fixed degree sequence, e.g., the one observed in a snapshot of a real network, and every such graph is equiprobable in the ensemble. The ensemble thus maximizes Gibbs entropy subject to the constraint that the degree sequence is fixed. Yet given a real network snapshot, one cannot usually trust its degree sequence as some ``ultimate truth'' for a variety of reasons, including measurement imperfections, inaccuracies, and incompleteness, noise and stochasticity, and most importantly, the fact that most real networks are dynamic both at short and long time scales, growing often by orders of magnitude over years~\cite{dhamdhere2011years,newman2001clustering,boccaletti2006complex,newman2010networks,barabasi2016network}.

These factors partly motivated the development of the soft configuration model (SCM), random graphs with a given expected degree sequence, first considered in~\cite{chung2002connected,chung2002average}, and later corrected in~\cite{park2004statistical,bianconi2008entropy,garlaschelli2008maximum,squartini2011analytical}, where it was shown that this correction yields a canonical ensemble of random graphs that maximize Gibbs entropy under the constraint that the expected degree sequence is fixed. In statistics, canonical ensembles of random graphs are known as exponential random graphs (ERGs)~\cite{holland1981exponential}. In~\cite{anand2009entropy,squartini2015breaking} it was shown that the sparse CM and SCM are not equivalent, but they are equivalent in the case of dense graphs~\cite{chatterjee2011random}. Yet the SCM still treats a given degree sequence as a fixed constraint, albeit not as a sharp but soft constraint. This constraint is in stark contrast with reality of many growing real networks, in which the degree of all nodes constantly change, yet the shape of the degree distribution and the average degree do not change, staying essentially constant in networks that grow in size even by orders of magnitude~\cite{dhamdhere2011years,newman2001clustering,boccaletti2006complex,newman2010networks,barabasi2016network}. These observations motivated the development of the hypersoft configuration model~\cite{caldarelli2002scale,boguna2003class,anand2014entropy,zuev2016hamiltonian}.

\subsection{Hypersoft configuration model (HSCM)}\label{sec:intro_hscm}

In the HSCM neither degrees nor even their expected values are fixed. Instead the fixed properties are the degree distribution and average degree. The HSCM with a given average degree and power-law degree distribution is defined by the exponential measure $\mu$ on the real line $\R$
\begin{align*}
\mu=e^{\gamma x},\quad x\in\R,
\end{align*}
where $\gamma>1$ is a constant, and by the Fermi-Dirac graphon $W:\R^2\to[0,1]$
\begin{equation}\label{eq:WFD}
W(x,y)=\frac{1}{e^{x+y}+1}.
\end{equation}
The volume-form measure $\mu$ establishes then probability measures
\begin{align*}
\mu_n = \frac{\mu|_{A_n}}{\mu(A_n)}= \gamma\,e^{\gamma(x-R_n)}
\end{align*}
on intervals $A_n=(-\infty,R_n]$, where
\begin{equation}\label{eq:intro_Rn}
  R_n = \frac{1}{2}\log\frac{n}{\beta^2\nu},\quad \beta=1-\frac{1}{\gamma},
\end{equation}
and $\nu>0$ is another constant. The constants $\gamma>1$ and $\nu>0$ are the only two parameters of the model. The HSCM random graphs of size $n$ are defined by $(W,A_n,\mu_n)$ via sampling $n$ i.i.d.\ points $x_1, \dots, x_n$ on $A_n$ according to measure $\mu_n$, and then connecting pairs of points $i$ and $j$ at sampled locations $x_i$ and $x_j$ by an edge with probability $W(x_i,x_j)$.

An alternative equivalent definition is obtained by mapping~$(W,A_n,\mu_n)$ to $(W_{I,n},I,\mu_I)$, where $I=[0,1]$, $\mu_I=1$, and
\begin{equation}\label{eq:W_n}
  W_{I,n}(x,y)=\frac{1}{\frac{n}{\beta^2\nu}\left(xy\right)^{\frac{1}{\gamma}}+1},\quad x,y\in I,
\end{equation}
In this definition, $x_i$s are $n$ i.i.d.\ random variables uniformly distributed on the unit interval $[0,1]$, and vertices $i$ and $j$ are connected with probability $W_{I,n}(x_i,x_j)$.

Yet another equivalent definition, perhaps the most familiar and most frequently used one, is given by $(W_{P,n},P,\mu_P)$, where interval $P=[\beta\nu,\infty)$, and measure $\mu_P$ on $P$ is the Pareto distribution
\begin{align*}
  \mu_P &= \gamma\left(\beta\nu\right)^\gamma x^{-\alpha},\quad x\in P,\\
  W_{P,n}(x,y) &= \frac{1}{\frac{\nu n}{xy}+1},\quad x,y\in P.
\end{align*}
In this Pareto representation, the expected degree of a vertex at coordinate~$x$ is proportional to~$x$~\cite{garlaschelli2008maximum,squartini2011analytical}.

Compared to the SCM where only edges are random variables while the expected degrees are fixed, the HSCM introduces another source of randomness---and hence entropy---coming from expected degrees that are also random variables. One obtains a particular realization of an SCM from the HSCM by sampling $x_i$s from their fixed distribution and then freezing them. Therefore the HSCM is a probabilistic mixture of canonical ensembles, SCM ERGs, so that one may call the HSCM a \emph{hypercanonical ensemble} given that latent variables $x$ in the HSCM are called \emph{hyperparameters} in statistics~\cite{evans2009probability}.

\subsection{Properties of the HSCM}

We prove in Theorem~\ref{thm:mixed_poisson_degrees_hscm} that the distribution of degrees $\D$ in the power-law HSCM ensemble defined above converges---henceforth \emph{convergence} always means the $n\to\infty$ limit, unless mentioned otherwise---to
\begin{equation}\label{eq:P(k)}
	\Prob{\D=k} = \gamma(\beta\nu)^\gamma\frac{\Gamma(k-\gamma,\beta\nu)}{k!}
	= \gamma(\beta\nu)^\gamma \frac{\Gamma(k - \gamma)}{\Gamma(k + 1)}\left(1 - P(k -\gamma, \beta \nu)\right),
\end{equation}
where $\Gamma(a,x)$ is the upper incomplete Gamma function, and $P(a,x)$ is the regularized lower incomplete Gamma function.
Since $P(a,x) \sim (ex/a)^a$ for $a \gg ex$, while $\Gamma(k - \gamma)/\Gamma(k + 1) \sim k^{-(\gamma + 1)}$ for $k \gg \gamma$, we get
\begin{equation}\label{eq:Pk_asymptotic}
	\Prob{\D=k} \sim \gamma(\beta\nu)^\gamma k^{-\alpha},\quad \alpha=\gamma+1.
\end{equation}
We also prove in Theorem~\ref{thm:average_degree_hscm} that the expected average degree in the ensemble converges to
\begin{equation}\label{eq:ED=nu}
  \Exp{\D}=\nu.
\end{equation}
That is, the degree distribution in the ensemble has a power tail with exponent $\alpha$, while the expected average degree is fixed to constant $\nu$ that does not depend on $n$, 
Fig.~\ref{fig:internet_vs_hscm}. 

\begin{figure}
	\centerline{
    \includegraphics[width=6in]{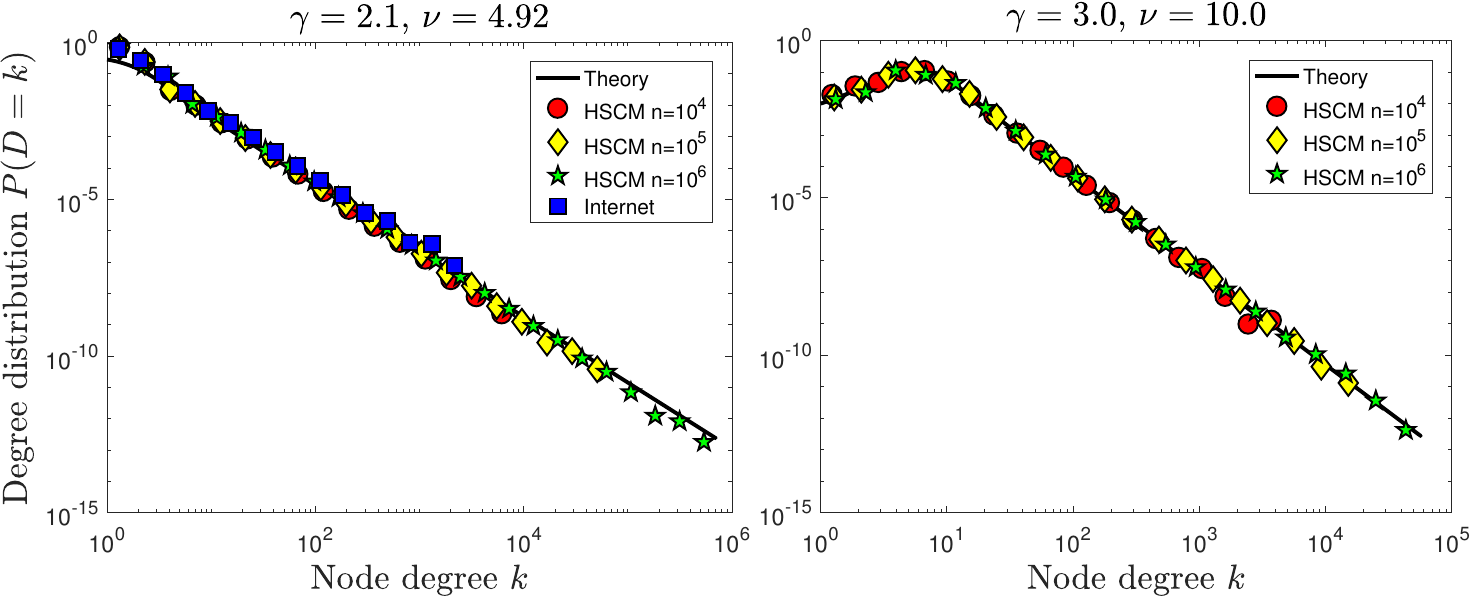}
    }
	\caption{{\bf Degree distribution in the HSCM, theory vs.\ simulations, and in the Internet.} The theory curve in the left panel is Eq.~\eqref{eq:P(k)} with $\gamma=1.1$ ($\alpha=2.1$) and $\nu=4.92$.
    The simulation data shown by symbols is averaged over $100$ random graphs for each graph size~$n$. All the graphs are generated according to the HSCM with the same $\gamma=1.1$ and $\nu=4.92$.
    The average degrees, averaged over $100$ random graphs, in the graphs of size $10^4$, $10^5$, and $10^6$, are $1.73$, $2.16$, and $2.51$, respectively.
    The Internet data comes from CAIDA's Archipelago measurements of the Internet topology at the Autonomous System level~\cite{claffy2009internet}.
    The number of nodes and the average degree in the Internet graph are $23,752$ and $4.92$.
    The right panel shows the theoretical degree distribution curve in Eq.~\eqref{eq:P(k)} with $\gamma=2.0$ and $\nu=10.0$
    versus simulations of $100$ random HSCM graphs of different sizes with the same $\gamma$ and $\nu$.
    The average degrees in the graphs of size $10^4$, $10^5$, and $10^6$, are $9.96$, $9.98$, and $10.0$, respectively.
    \label{fig:internet_vs_hscm}}
\end{figure}

\subsection{Unbiasedness and the maximum-entropy requirement}

While the average degree in the HSCM converges to a constant, and the degree distribution converges to a power law, the power-law HSCM is certainly just one of an infinite number of other models that possess these two properties. One example is random hyperbolic graphs~\cite{krioukov2010hyperbolic} that also have constant average degree and power-law degree distribution, but have larger numbers of triangles, non-zero clustering in the limit. Is the HSCM an unbiased model of random power-law graphs with constant average degree? That is, are the HSCM random graphs characterized by only these two properties and no others? Or colloquially, is the HSCM the model of ``maximally random'' power-law graphs with a constant average degree. This question can be formally answered by checking whether the HSCM satisfies the maximum-entropy requirement.

A discrete distribution $p_i$, $i=1,2,\ldots$, is said to satisfy the maximum-entropy requirement subject to constraints $\sum_i p_i f_{ir} = \bar{f}_r$, $r=1,2,\ldots$, where $f_r$s are some real functions of states~$i$, and $\bar{f}_r$s are a collection of real numbers, if the Gibbs/Shannon entropy of the distribution $S=-\sum_i p_i \log p_i$ is maximized subject to these constraints~\cite{jaynes1957information}. This entropy-maximizing distribution is known to be always unique, belonging to the exponential family of distributions, and it can be derived from the basic consistency axioms: uniqueness and invariance with respect to a change of coordinates, system independence and subset independence~\cite{shore1980axiomatic,tikochinsky1984consistent,skilling1988axioms}. Since entropy is the unique measure of information satisfying the basic requirements of continuity, monotonicity, and system/subset independence~\cite{shannon1948mathematical}, the maximum-entropy requirement formalizes the notion of encoding into the probability distribution $p_i$ describing a stochastic system, all the available information about the system given to us in the form of the constraints above, and \emph{not} encoding any other information not given to us. Since the entropy-maximizing distribution is unique, any other distribution necessarily but possibly implicitly introduces biases by encoding some additional ad-hoc information, and constraining some other system properties, concerning which we are not given any information, to some ad-hoc values. Clearly, such uncontrolled information injection into a model of a system may affect the predictions one may wish to make about the system using the model, and indeed it is known that given all the available information about a system, the predictive power of a model that describes the system is maximized by the maximum-entropy model~\cite{shore1980axiomatic,tikochinsky1984consistent,skilling1988axioms}. Perhaps the best illustration of this predictive power is the predictive power of equilibrium statistical mechanics, which can be formulated almost fully in terms of the maximum-entropy principle~\cite{jaynes1957information}.

To illustrate the maximum-entropy requirement in application to random graphs, suppose we are to define a random graph ensemble, and the \emph{only} available information is that these random graphs must have $n$ nodes and $m$ edges. From the purely probabilistic perspective, any random graph ensemble satisfying these constraints---random $m$-stars or $m$-cycles, for instance, if $m<n$---would be an equally good one. Yet there is only one unique ensemble that satisfies not only these constraints but also the maximum-entropy requirement. This ensemble is $\mathcal{G}_{n,m}$ because in $\mathcal{G}_{n,m}$ \emph{any} graph with $n$ nodes and $m$ edges is equally likely, so that the probability distribution on the set of all graphs with $n$ nodes and $m$ edges is uninform, and without further constraints, the uniform distribution is the maximum-entropy distribution on the state space, which in this case is $i=1,\ldots,{{n\choose2}\choose m}$, the number of such graphs. Random $m$-stars or $m$-cycles, while satisfying the constraints, inject into the model, in this case explicitly, additional information about the graph structure that was not given to us. Clearly, predictions based on random $m$-stars versus $\mathcal{G}_{n,m}$ may be very different, as the first model trivially predicts that $m$-stars occur with probability $1$, while they appear with a nearly zero probability in $\mathcal{G}_{n,m}$ if $n$ and $m$ are large.

A slightly less trivial example is the SCM. In this case the given information is that the expected degrees of nodes $i=1,\ldots,n$ must be $k_i\in\R_+$, and the state space is all the $2^{n\choose2}$ graphs on $n$ nodes. As shown in~\cite{park2004statistical,bianconi2008entropy,garlaschelli2008maximum,squartini2011analytical}, the unique entropy-maximizing ensemble satisfying these constraints is given by random graphs in which nodes $i$ and $j$ are connected with probabilities $p_{ij}=1/\left(kn/\kappa_i\kappa_j+1\right)$, where $k=\sum_i k_i$, and $\kappa_i$s are the unique solution of the system of $n$ equations $\sum_j p_{ij}=k_i$. The popular Chung-Lu model~\cite{chung2002connected,chung2002average} is different in that the connection probability there is $p_{ij}^{CL}=\min\left(k_ik_j/kn,1\right)$, which can be thought of as a classical-limit approximation to the entropy-maximizing Fermi-Dirac $p_{ij}$ above. While the CL ensemble also satisfies the desired constraints $\sum_j p_{ij}^{CL}=k_i$ (albeit only for sequences $k_i$ s.t.\ $k_ik_j/kn\leq 1$), it does not satisfy the maximum-entropy requirement, so that it injects, in this case implicitly, some additional information into the ensemble, constraining some undesired properties of graphs in the ensemble to some ad-hoc values. Since the undesired information injection is implicit in this case, it may be quite difficult to detect and quantify all the biases introduced into the ensemble.

\subsection{Main results}

The main result of this paper is the proof in Theorem~\ref{thm:main_result} that the HSCM is unbiased, that is, that the HSCM random graphs maximize the Gibbs entropy of random graphs whose degree distribution and average degree converge to~(\ref{eq:P(k)},\ref{eq:ED=nu}).

The first difficulty that we face in proving this result is how to properly formulate the entropy-maximization problem under these constraints. Indeed, we are to show that the probability distributions $P_n$ that the HSCM defines on the set of $n$-sized graphs $G_n$ maximizes the graph entropy
\begin{align*}
	\mathcal{S}[P'_n] = - \sum_{G_n}P'_n(G_n) \log P'_n(G_n)
\end{align*}
across all the distributions $P'_n$ that define random graph ensembles with the degree distributions and average degrees converging to~(\ref{eq:P(k)},\ref{eq:ED=nu}). These constraints are quite different than the SCM constraints, for example, because for any fixed $n$, we do not have a fixed set of constraints or sufficient statistics. Instead of introducing such sufficient statistics, e.g., expected degrees converging to a desired Pareto distribution, and proceeding from there, we show in Section~\ref{sec:definitions} that the problem of graph entropy maximization under these constraints is equivalent to a graphon entropy maximization problem~\cite{krioukov2016clustering}, i.e., to the problem of finding a graphon $W$ that maximizes graphon entropy
\begin{align*}
	\sigma_n[W'] = \iint_{A_n^2} H\left(W'(x,y)\right)\,d\mu_n(x)\,d\mu_n(y),
\end{align*}
where $H(p)=-p \log p - (1-p) \log(1-p)$ is the entropy of a Bernoulli random variable with success probability $p$, across all the graphons $W'$ that satisfy the constraint
\begin{align*}
	(n-1)\int_{A_n} W'(x,y)\,d\mu_n(y) = \kappa_n(x),
\end{align*}
where $\kappa_n(x) \approx \sqrt{\nu n} \, e^{-x}$ is the expected degree of a node at coordinate $x$ in the power-law HSCM. We then prove in Proposition~\ref{prop:max_entropy_graphon} that the unique solution to this graphon entropy maximization problem is given by the Fermi-Dirac graphon~$W$ in~\eqref{eq:WFD}. The fact that the Fermi-Dirac graphon is the unique solution to this graphon entropy maximization problem is a reflection of the basic fact in statistical physics that the grand canonical ensemble of Fermi particles, which are edges of energy $x+y$ in our case, is the unique maximum-entropy ensemble with fixed expected values of energy and number of particles~\cite{kapur1989maximum}, in which the probability to find a particle at a state with energy $x+y$ is given by~\eqref{eq:WFD}.

Yet the solutions to the graph and graphon entropy maximization problems yield equivalent random graph ensembles only if the rescaled graph entropy $\mathcal{S}^\ast[P_n] = \mathcal{S}[P_n]/\binom{n}{2}$ converges to the graphon entropy $\sigma_n[W]$. Here we face another difficulty that since our ensembles are sparse, both $\mathcal{S}^\ast[P_n]$ and $\sigma_n[W]$ converge to zero, so that we are actually to prove that the two entropies converge to each other faster than either of them converges to zero. To this end we prove in Theorems~\ref{thm:convergence_graphon_entropy} and~\ref{thm:convergence_graph_entropy} that both the graphon and graph entropies converges to zero as $\sigma_n[W],\mathcal{S}^\ast[P_n]\sim \nu \log(n) / n$. The key result then, also in Theorem~\ref{thm:convergence_graph_entropy}, is the proof that if divided by the scaling factor of $\log(n) / n$, the difference between the graphon and graph entropies vanishes in the limit, $\lim_{n\to\infty} (n/\log n)|\mathcal{S}^\ast[P_n]-\sigma_n[W]|\to0$, meaning that the two entropies do indeed converge to each other faster than to zero.

The combination of graphon~\eqref{eq:WFD} being the entropy maximizer, and the convergence of the rescaled graph entropy to the entropy of this graphon, implies the main result in Theorem~\ref{thm:main_result} that the power-law HSCM is a graph entropy maximizer subject to the degree distribution and average degree constraints~(\ref{eq:P(k)},\ref{eq:ED=nu}).

\subsection{Exchangeability and projectivity}

In addition to the natural, dictated by real-world network data, requirements of constant, i.e., independent of graphs size $n$, average degree and power-law degree distribution, as well as the maximum-entropy requirement, dictated by the basic statistical considerations, a reasonable model of real networks must also satisfy two more requirements: exchangeability and projectivity.

Exchangeability takes care of the fact that node labels in random graph models are usually meaningless. Even though node labels in real networks often have some network-specific meaning, such as autonomous system numbers in the Internet~\cite{dhamdhere2011years}, node labels in random graph models can be, and usually are, random integer indices $i=1,2,\ldots$. A random graph model is exchangeable if for any permutation $\sigma$ of node indices $i$, the probabilities of any two graphs $G$ and $G_\sigma$ given by adjacency matrices $G_{i,j}$ and $G_{\sigma(i),\sigma(j)}$ are the same, $P(G)=P(G_\sigma)$~\cite{aldous1981representations,diaconis2008graph}.

A random graph model is projective if there exists a map $\pi_{n\mapsto n'}$ from graphs of size $n$ to graphs of size $n'<n$ such that the probability of graphs in the model satisfies $P(G_{n'})=P(\pi_{n\mapsto n'}(G_n))$~\cite{kallenberg2002foundations,shalizi2013consistency}. If this condition is satisfied, then it is easy to see that the same model admits a dual formulation as an equilibrium model of graphs of a fixed size, or as a growing graph model~\cite{krioukov2013duality}. If this requirement is not satisfied, then as soon as one node is added to a graph, e.g., due to the growth of a real network that this graph represents, then the resulting bigger graph is effectively sampled from a different distribution corresponding to the model with different parameters, necessarily affecting the structure of its existing subgraphs, a clearly unrealistic scenario. As the simplest examples, $\G_{n,p}$ is projective (the map $\pi_{n\mapsto n'}$ simply selects any subset of $n$ nodes consisting of $n'$ nodes), but $\G_{n,k/n}$ with constant $k$ is not. In the first case, one can realize $\G_{n,p}$ by growing graphs adding nodes one at a time, and connecting each new node to all existing nodes with probability $p$, while in the second case such growth is impossible since the existing edges in the growing graphs must be removed with probability $1/n$ for the resulting graphs to be samples from $\G_{n,k/n}$ for each~$n$.

The HSCM random graphs are manifestly exchangeable as any graphon-based ensemble~\cite{lovasz2006limits,janson2013graphons}. Here we note that the fact that these graphs are both sparse and exchangeable is \emph{not} by any means in conflict with the Aldous-Hoover theorem~\cite{aldous1981representations,hoover1979relations} that states that the limit graphon, mapped to the unit square, of any exchangeable sparse graph family is necessarily zero. Indeed, if mapped to unit square, the limit HSCM graphon $W_{I,n}$~\eqref{eq:W_n} is zero as well. We also note that the convergence of $W_{I,n}$ to zero does not mean that the ensemble converges to infinite empty graphs. In fact, the expected degree distribution and average degree in the ensemble converge to~(\ref{eq:P(k)},\ref{eq:ED=nu}) in the limit, as stated above.

If $\gamma = 2$, the HSCM ensemble is also projective, but only with a specific labeling of nodes breaking exchangeability. This can be seen by observing that the density of points on intervals $A_n$, $\delta_n=n/\mu(A_n)=\gamma\left(\beta^2\nu\right)^{\gamma/2}n^{1-\gamma/2}$, and consequently on the whole real line $\R$ in the limit, is constant $\delta=\nu/2$ if $\gamma=2$. In this case, the HSCM can be equivalently defined as a model of growing labeled graphs as follows: for $n=1,2,\ldots$, the location $x_n$ of new node $n$ belongs to $A_n$'s increment, $x_n\in B_n = A_n\setminus A_{n-1}$ ($A_0=\emptyset$), and sampled from $\mu$ restricted to this increment, i.e., from the probability measure $\tilde{\mu}_n=\mu|_{B_n}/\mu(B_n)=\gamma\,e^{\gamma(x-R_n+R_{n-1})}$. Having its location sampled, new node $n$ then connects to all existing nodes $i=1\ldots n-1$ with probability given by~\eqref{eq:WFD}.

This growing model is equivalent to the original equilibrium HSCM definition in Section~\ref{sec:intro_hscm} only asymptotically. However, the exact equivalence, for each $n$, between the equilibrium HSCM with ordered $x_i$s, $x_i<x_{i+1}$, and its growing counterpart can be also achieved by ensuring that the joint distribution of $x_i$s is exactly the same in both cases, using basic properties of Poisson point processes~\cite{krioukov2013duality}. Specifically, the equilibrium definition in Section~\ref{sec:intro_hscm} must be adjusted by making the right boundary $R_n$ of interval $A_n$ not a fixed function of $n$~\eqref{eq:intro_Rn}, but a random variable $R_n=(1/2)\log(2V_n)$, where $V_n$ is a random variable sampled from the Gamma distribution with shape $n$ and rate $\delta=\nu/2$. Node $n$ is then placed at random coordinate $x_n=R_n$, while the coordinates of the rest of $n-1$ nodes are sampled from probability measure $\mu_n$---measure $\mu$ restricted to the random interval $A_n=(-\infty,R_n]$---and then labeled in the increasing order of their coordinates. The growing model definition must be also adjusted: the coordinate $x_{n+1}$ of the $n+1$'th node is determined by $v_{n+1}=v_n+V$, where $v_0=0$, $v_i=(1/2)e^{2x_i}$, and $V$ is a random variable sampled from the exponential distribution with rate $\delta=\nu/2$. One can show that coordinates $x_i$, both for finite and infinite $n$, in both the equilibrium and growing HSCM models defined this way, are equivalent realizations of the same Poisson point process on $\R$ with measure $\mu$ and rate $\delta$, converging to the binomial sampling with $R_n$ fixed to~\eqref{eq:intro_Rn}~\cite{krioukov2013duality}.

The projective map $\pi_{n\mapsto n'}$ in the projectivity definition above, simply maps graphs $G_n$ to their subgraphs induced by nodes $i=1\ldots n'$. We note that even though the growing HSCM is not exchangeable since it relies on labeling of nodes in the increasing order of their coordinates, it is nevertheless equivalent to the equilibrium HSCM with this ordered labeling, because the joint distribution of node coordinates, and the linking probability as a function of these coordinates are the same in both the equilibrium and growing HSCM definitions~\cite{krioukov2013duality}. This observation suggests that there might exist a less trivial projective map such that the HSCM is both projective and exchangeable at the same time.

\subsection{Other remarks}

We note that thanks to its projectiveness, the power-law HSCM was shown in~\cite{zuev2016hamiltonian} to be equivalent to a soft version of preferential attachment, a model of growing graphs in which new nodes connect to existing nodes with probabilities proportional to the expected degrees of existing nodes. It is well-known that similar to the HSCM, the degree distribution and average degree in graphs grown according to preferential attachment do not essentially change either as graphs grow~\cite{dorogovtsev2000structure,krapivsky2000connectivity}. If $\gamma=2$, then the equivalence between the HSCM and soft preferential attachment is exact. If $\gamma\neq2$, then the HSCM, even with ordered labeling, is not equivalent to soft preferential attachment, but it is equivalent to its adjusted version with a certain rate of (dis)appearance of edges between existing vertices~\cite{zuev2016hamiltonian}. 

We also note that the HSCM is the zero-clustering limit~\cite{aldecoa2015hyperbolic} of random hyperbolic graphs~\cite{krioukov2010hyperbolic}, where the $\gamma=2$ case corresponds to the uniform density of points in the hyperbolic space $\mathbb{H}^d$, and where $x_i$s are the radial coordinates of nodes $i$ in the spherical coordinate system of the hyperboloid model of $\mathbb{H}^d$. These coordinates can certainly not be negative, but the expected fraction of nodes with negative coordinates in the HSCM is negligible: $\mu_n(\R_-)=\left(\beta^2\nu/n\right)^{\gamma/2}\to0$. In the zero-clustering limit, the angular coordinates of nodes in $\mathbb{H}^d$ are ignored in the hyperbolic graphon~\cite{krioukov2010hyperbolic}, which becomes equivalent to~\eqref{eq:WFD}.

As a final introductory remark we note that among the rigorous approaches to sparse exchangeable graphs, the HSCM definition is perhaps closest to graphon processes and graphexes in~\cite{caron2014sparse,veitch2015class,borgs2016sparse}. In particular in~\cite{borgs2016sparse}, where the focus is on graph convergence to well-defined limits, two ensembles are considered. One ensemble, also appearing in~\cite{veitch2015class}, is defined by any graphon $W:\R_+^2\to[0,1]$ and any measure $\mu$ on $\R_+$ ($\R_+$ can be replaced with any measure space). Graphs of a certain expected size, which is a growing function of time $t>0$, are defined by sampling points as the Poisson point process on $\R_+$ with intensity $t\mu$, then connecting pairs of points with the probability given by $W$, and finally removing isolated vertices. The other ensemble is even more similar to the HSCM. It is still defined by $W$ and $\mu$ on $\R_+$, but the location of vertex $n$ on $\R_+$ is sampled from $\mu_n=\mu|_{A_n}/\mu(A_n)$, where $A_n$s are finite-size intervals growing with $n$ whose infinite union covers the whole $\R_+$. The latter ensemble is not exchangeable, but both ensembles are shown to converge to properly stretched graphons defined by $W$, yet only if the expected average degree grows to infinity in the limit. The HSCM definition is different---in particular all $n$ vertices of $n$-sized graphs are sampled from the same $\mu_n$---ensuring exchangeability, and allowing for an explicit control of the degree distribution and average degree, which can be constant, but making the problem of graph convergence difficult. We do not further discuss graph convergence here, leaving it, as well as the generalization of the results to arbitrary degree distributions, for future publications.

\subsection{Paper organization}

In the next Section~\ref{sec:definitions} we first review in more detail the necessary background information and provide all the required definitions. In Section~\ref{sec:results} we formally state all the results in the paper, while Section~\ref{sec:proofs} contains all the proofs of these results.

\section{Background information and definitions}\label{sec:definitions}

\subsection{Graph ensembles and their entropy}

A graph ensemble is a set of graphs $\mathcal{G}$ with probability measure $P$ on $\mathcal{G}$. The
Gibbs entropy of the ensemble is
\begin{equation}\label{eq:def_gibbs_entropy}
	\mathcal{S}[P] = - \sum_{G \in \mathcal{G}}P(G) \log P(G)
\end{equation}
Note that this is just the entropy of the random variable $G$ with respect to the probability measure $P$.
When $G_n$ is a graph of size $n$, sampled from $\mathcal{G}$ according to measure $P$, we write $\mathcal{S}[G_n]$
instead of $\mathcal{S}[P]$. Given a set of constraints, e.g., in the form of graph properties fixed to given values,
the maximum-entropy ensemble is given by $P^*$ that maximizes $\S[P]$ across all measures $P$ that satisfy the constraints.
These constraints can be either sharp (microcanonical) or soft (canonical), satisfied either exactly or on average, respectively.
The simplest example of the constrained graph property is the number of edges, fixed to $m$, in graphs of size $n$.
The corresponding microcanonical and canonical maximum-entropy ensembles are $\G_{n,m}$ and $\G_{n,p}$ with $p=m/{n\choose2}$, respectively.
The $P^*$ is respectively the uniform and exponential Boltzmann distribution $P(G)=e^{-H(G)}/Z$ with Hamiltonian $H(G)=\lambda m(G)$, where
$m(G)$ is the number of edges in graph $G$, and the Lagrange multiplier $\lambda$ is given by $p=1/\left(e^\lambda+1\right)$~\cite{park2004statistical}.

When the constraints are given by the degrees of nodes, instead of the number of edges, we have the following characterization of
the microcanonical and canonical ensemble.

\subsubsection{Maximum-entropy graphs with a given degree sequence (CM)}
Given a degree sequence ${\bf d}_n = d_1\ldots d_n$, the microcanonical ensemble of graphs that have this degree
sequence is the configuration model (\texttt{CM})~\cite{bender1978asymptotic,molloy1995critical}.
The entropy-maximizing $P^*$ is uniform on the set of all graphs that have degree sequence ${\bf d}_n$.

\subsubsection{Maximum-entropy graphs with a given expected degree sequence (SCM)}\label{sec:SCM}
If the sharp CM constraints are relaxed to soft constraints, the result is the canonical ensemble of the soft configuration
model (\texttt{SCM})~\cite{bianconi2008entropy,garlaschelli2008maximum}. Given an expected degree
sequence ${\bf k}_n$, which in contrast to CM's ${\bf d}_n$, does not have to be a graphical sequence
of non-negative integers, but can be any sequence of non-negative real numbers, the SCM is defined by connecting nodes $i$ and $j$ with probabilities
\begin{align}\label{eq:def_scm_lambda}
  p_{ij} &= \frac{1}{e^{\lambda_i + \lambda_j} + 1}\text{, where Lagrange multipliers $\lambda_i$ are the solution of} \\
  k_i &= \sum_{i<j} p_{ij}, \quad i=1\ldots n.
\end{align}
The entropy-maximizing $P^*$ is the Boltzmann distribution
with Hamiltonian $H(G)=\sum_i \lambda_i d_i(G)$, where $d_i(G)$ is the degree of node $i$ in graph $G$~\cite{bianconi2008entropy,garlaschelli2008maximum}.

\subsection{Sparse graphs with a given degree distribution}

Let $p(k)$, $k=0,1,2,\ldots$, be a 
probability density function with finite mean. Denote by $\mathscr{D}$ the corresponding random variable,
and consider a sequence of graph ensembles that maximize Gibbs entropy under the
constraint that for all $k$
\begin{equation}\tag{C1}\label{eq:convergence_degree_hscm_constraint}
	\lim_{n \to \infty} \Prob{\mathscr{D}_n = k} = \Prob{\mathscr{D} = k},
\end{equation}
where $\mathscr{D}_n$ is the degree of a uniformly chosen node in the ensemble of graphs of size $n$. In other words, this is a maximum-entropy ensemble of graphs 
whose degree distribution converges to $p(k)$.

In addition to the degree distribution constraint, we also want our graphs to be sparse. The most common definition of \emph{sparseness} seems to be
that the number of edges is $o(n^2)$, so that the expected average degree can be unbounded. In contrast, here we use the term \emph{sparse} to mean that
the expected degree converges to the expected value of $\mathscr{D}$:
\begin{equation}\tag{C2}\label{eq:convergence_expected_degree_constraint_hscm}
	\lim_{n \to \infty} \Exp{\mathscr{D}_n} = \Exp{\mathscr{D}} := \nu.
\end{equation}
Constraint~\eqref{eq:convergence_expected_degree_constraint_hscm} implies that the number of edges is $O(n)$. We note that, in general, 
\eqref{eq:convergence_expected_degree_constraint_hscm} does not follow from \eqref{eq:convergence_degree_hscm_constraint}, since convergence in distribution
does not necessarily imply convergence in expectation.

We also note that constraints~(\ref{eq:convergence_degree_hscm_constraint},\ref{eq:convergence_expected_degree_constraint_hscm})
are neither sharp nor soft, since they deal only with the $n \to \infty$ limits of the degree distribution and expected degree.
We call these constraints \emph{hypersoft}, since as we will see below, random graphs satisfying these constraints can
be realized as random graphs with Lagrange multipliers that are not parameters but \emph{hyperparameters} in the statistics
terminology~\cite{evans2009probability}.
 
\subsection{Maximum-entropy graphs with hypersoft constraints}

Similar to the case of random graphs with a given (expected) degree sequence, we are to determine the distribution $P$ that satisfies 
\eqref{eq:convergence_degree_hscm_constraint} and \eqref{eq:convergence_expected_degree_constraint_hscm}, and maximizes the
Gibbs entropy. However, this task poses the question of what it means to maximize entropy under these limit constraints.
In particular, unlike an ensemble of graphs with a given (expected) degree sequence, we are no longer dealing with a set of graphs of 
fixed size, but with a sequence of graphs $(G_n)_{n \ge 1}$ of varying sizes.  
To answer this question, and to give a proper definition of entropy maximization under hypersoft constraints, we consider graphon-based
ensembles of graphs.

\subsubsection{Graphon-based graph ensembles}\label{sec:graphon-ensembles}

In the simplest case, a graphon is a symmetric integrable function $W : [0, 1]^2 \to [0,1]$. Graphons, or more precisely graphon
equivalence classes, consisting of all functions $W_\sigma$ such that $W(x,y)=W_\sigma(\sigma(x),\sigma(y))$ under all measure preserving
transformations $\sigma:[0,1]\to[0,1]$, are well-defined limits of dense graph families~\cite{lovasz2012large,janson2013graphons}.
One can think of the interval $[0,1]$ as the continuum
limit of node indices $i$, and of $W$ as the limit of graphs' adjacency matrices. Equivalently,
$W(x,y)$ is the probability that there exists the edge between ``nodes'' $x$ and $y$.
Graphons are an application to graphs of a class of earlier results on exchangeable arrays in statistics~\cite{aldous1981representations,aldous1983exchangeability},
and are better known as the \emph{connection probability} in random graphs with latent parameters in sociology~\cite{mcfarland1973social,faust1988comparison,mcpherson1991evolution,hoff2002latent}
and network science~\cite{caldarelli2002scale,boguna2003class}, also known in graph theory as \emph{inhomogeneous random graphs}~\cite{bollobas2007phase}.
Here we use the term \emph{graphon} to refer to any symmetric function $W : A \times A \to [0,1]$, for some $A \subseteq \mathbb{R}$.

Let $\mu$ be a probability measure on $A \subseteq \mathbb{R}$ and $W$ a graphon. Then the standard graphon-based ensemble
$\mathcal{G}_{W, \mu, A}$ known as $W$-random graphs~\cite{lovasz2006limits} is the ensemble of random $n$-sized graphs $G_n$ defined by first
i.i.d.\ sampling $n$ node coordinates ${\bf x}_n  = x_1, \dots, x_n \in A$
according to $\mu$, and then connecting every node pair $(i,j)$, independently, with probability $W(x_i,x_j)$. 

To be able to satisfy the hypersoft constraints we generalize this ensemble as follows. Let $\mu$ be a measure on $\mathbb{R}$, $W$ a graphon, and let 
$\mathbf{A}=(A_1,A_2,\ldots)$, $A_n \subseteq A_{n + 1}$, be an infinite sequence of growing subsets of $\mathbb{R}$ such that $\bigcup_{n \ge 1} A_n = \mathbb{R}$ and $\mu(A_n) < \infty$. 
We then define the graphon ensemble $\mathcal{G}_{W, \mu,\mathbf{A}} = (\mathcal{G}_{W, \mu_n, n})_{n \ge 1}$ to be random graphs $(G_n)_{n \ge 1}$ 
defined by the graphon $W$ and measures
\begin{align*}
	\mu_n = \frac{\mu|_{A_n}}{\mu(A_n)},
\end{align*}
which are the probability measures on $A_n$ associated with $\mu$. To sample a graph $G_n$ from this ensemble, $n$ i.i.d.\ points
$x_i\in A_n$ are first sampled according to measure $\mu_n$, and then pairs of nodes $i$ and $j$ are connected by an edge with probability $W(x_i,x_j)$.

We remark that if $A_n = A = [0,1]$ and $\mu$ is the uniform measure,
we are in the most classical settings of $W$-random graphs~\cite{lovasz2006limits}. In the case of an arbitrary $\mathbf{A}$ and measure $\mu$,
our random graph ensemble is similar to a model considered recently in~\cite{borgs2016sparse}, Section 2.4.
There, a growing graph construction is considered where $G_{n}$ is created from $G_{n - 1}$ by sampling the coordinate $x_n$ of node $n$ according to measure $\mu_n$,
and then connecting it to all existing nodes $i$, independently with probability $W(x_n, x_i)$. The main difference between such sampling and the one above is that
in the former case, the coordinates $x_i$ of different nodes $i=1\ldots n$ are sampled from different measures $\mu_i$, thus breaking exchangeability,
while in the latter case, all $x_i$s are sampled from the same measure $\mu_n$.

\subsubsection{Bernoulli and graphon entropies}\label{sec:bernoulli_graphon_entropies}

Given the coordinates ${\bf x}_n$, edges in our random graphs are independent Bernoulli random variables, albeit with different success probabilities $W(x_i, x_j)$. 
The conditional Bernoulli entropy of random graphs with fixed coordinates ${\bf x}_n$ is thus
\begin{align*}
	\mathcal{S}[G_n| {\bf x}_n] = \sum_{i < j} H\left(W(x_i, x_j)\right),
\end{align*}
where $H$ is the Bernoulli entropy,
\begin{equation}\label{eq:def_p_entropy}
	H(p) = -p\log p - (1-p) \log (1-p), \quad 0 \le p \le 1.
\end{equation}

The \emph{graphon entropy}~\cite{aldous1983exchangeability,janson2013graphons,hatami2013graph,chatterjee2011large,chatterjee2013estimating,radin2015singularities}
of $W$ with respect to $\mu$ on $A$ is defined as
\begin{equation}\label{eq:def_graphon_entropy}
	\sigma[W, \mu, A] = \iint_{A^2} H\left(W(x,y)\right)\,d\mu(x)\,d\mu(y),
\end{equation}
which we write as $\sigma[W, \mu]$, when $\mu$ has support $A$ and no confusion arises. In addition, if $G_n$ is a graph in the graphon ensemble $\mathcal{G}_{W, \mu_n, n}$,
we write $\sigma[G_n]$ for $\sigma[W, \mu_n, A_n]$.

Since for any two discrete random variables $X$ and $Y$, the expectation of the conditional entropy of $X$ given $Y$ is a lower bound for $X$'s entropy, 
$\mathcal{S}[X] \ge \Exp{\mathcal{S}[X|Y]}$, the graphon entropy definition implies
\begin{equation}\label{eq:gibbs-bernoulli-graphon-entropy}
	\mathcal{S}[G_n] \ge \Exp{\mathcal{S}[G_n| {\bf x}_n]} = \binom{n}{2} \Exp{H(W(x,y))} = \binom{n}{2} \sigma[G_n],
\end{equation}
The graphon entropy is thus a lower bound for the rescaled Gibbs entropy defined as
\begin{equation}\label{eq:rescaled_gibbs_entropy}
\mathcal{S}^\ast[G_n] = \frac{\mathcal{S}[G_n]}{\binom{n}{2}}. 
\end{equation}

Before we give our definition of the maximum-entropy ensemble of sparse graphs, it is instructive to consider the case of dense graphs.

\subsubsection{Dense maximum-entropy graphs with a given degree distribution}\label{sec:dense_HSCM}

Consider a sequence $({\bf d}_n)_{n \ge 1}$ of degree sequences ${\bf d}_n=(d_1\ldots d_n)_n$, $(d_i)_n\geq (d_{i+1})_n$, for which there exist constants $0 < c_1 \le c_2$ such
that $c_1 n \le (d_i)_n \le c_2 n$ for any $(d_i)_n \in {\bf d}_n$. Now let $(\G_n)_{n \ge 1}$ be a sequence of microcanonical ensembles of random
graphs, that is, CMs, defined by ${\bf d}_n$. 
If there exists a function
$f : [0,1] \to (0,1)$ such that for any $c\in(0,1)$
\begin{equation}\label{eq:dense-P-limit}
	\lim_{n \to \infty} \Prob{\frac{\mathscr{D}_n}{n} \leq c} = \Prob{f(U) \leq c}, 
\end{equation}
where $\mathscr{D}_n$ is the degree of a random node in random graph $G_n\in\G_n$ or equivalently, a uniformly sampled $(d_i)_n\in{\bf d}_n$,
and $U$ is a uniform random variable on $[0,1]$,
then it was proven in~\cite{chatterjee2011random} that the limit of the CM sequence $(\G_n)_{n \ge 1}$ is given by the graphon
\begin{equation}\label{eq:dense_max_entropy_graphon}
	W(x,y) = \frac{1}{e^{g(x) + g(y)}+1},
\end{equation}
where $g(x)$ is such that
\begin{equation}\label{eq:dense_graphon_entropy_constraint}
	f(x) = \int_0^1 W(x,y) \, dy = \int_A W_A(x,y)\,d\mu(y) = \int_A \frac{1}{e^{x+y}+1}\,d\mu(y),
\end{equation}
where $A$ is the image of $g$ and $\mu=g^{-1}$ (functions $f$ and $g$ are continuous and strictly increasing almost everywhere on $[0,1]$~\cite{chatterjee2011random}).

Some important observations are in order here. First, we note that \eqref{eq:dense-P-limit} is very similar to \eqref{eq:convergence_degree_hscm_constraint}, with the exception that \eqref{eq:dense-P-limit} implies that the 
degrees of all nodes are $O(n)$, so that the graphs are dense. In particular, in our graphs we have that $\D_n = o(n)$
so that $\D_n/n \to 0$.

Second, consider the problem of maximizing the graphon entropy under the constraint
given by \eqref{eq:dense_graphon_entropy_constraint}. We will show in Proposition \ref{prop:max_entropy_graphon} that
the solution to this problem is given by \eqref{eq:dense_max_entropy_graphon}, where $g$ is defined by
\eqref{eq:dense_graphon_entropy_constraint}. Hence the graphon~\eqref{eq:dense_max_entropy_graphon} obtained in \cite{chatterjee2011random} maximizes the 
graphon entropy under the constraint~\eqref{eq:dense_graphon_entropy_constraint} imposed by the limit $f$ of the sequence of rescaled degree sequences $({\bf d}_n/n)_{n \ge 1}$.

Third, Theorem D.5 in~\cite{janson2013graphons} states that in dense $W$-random graph ensembles $\G_{W,\mu,A}$ 
\begin{equation}\label{eq:convergence_dense_gibbs_entropy}
	\lim_{n \to \infty} \mathcal{S}^\ast[G_n] = \sigma[W, \mu, A],
\end{equation}
meaning that the rescaled Gibbs entropy of $W$-random graphs converges to the graphon entropy of~$W$.

Given $f$, this result suggests to
consider the family of $W$-random graph ensembles $\tilde{G}_n$ defined by the Fermi-Dirac graphon~\eqref{eq:dense_max_entropy_graphon}
with $g$ given by~\eqref{eq:dense_graphon_entropy_constraint}, which we call the dense hypersoft configuration model (\texttt{HSCM}).
The distribution of rescaled degrees $\D_n/n$ in these HSCM
graphs converges to $f$, cf.~\eqref{eq:dense-P-limit}, 
while the limit of these graphs is also $W$ since the limit of any dense $W$-random graphs is $W$~\cite{borgs2008convergent}. That is, the limit of HSCM
ensembles $\tilde{G}_n$ and the limit of CM ensembles $\G_n$ with any $({\bf d}_n/n)_{n\geq1}$ converging to $f$, are the same Fermi-Dirac graphon $W$~\eqref{eq:dense_max_entropy_graphon}.
Since the two ensembles have the same graphon limit, their rescaled Gibbs entropies converge to the same value, equal, thanks to~\eqref{eq:convergence_dense_gibbs_entropy}, to the graphon entropy,
even though for any finite $n$ the two ensembles are quite different.

Fourth, if we replace the sequence of degree sequences $({\bf d}_n)_{n \ge 1}$ with a sequence of expected degree
sequences $({\bf k}_n)_{n \ge 1}$ converging upon rescaling to $f$, and then replace the sequence of CMs with
the corresponding sequence of SCMs, then the limit of this SCM sequence is the same graphon $W$~\eqref{eq:dense_max_entropy_graphon}~\cite{chatterjee2011random},
so that the rescaled SCM Gibbs entropy also converges to the same graphon entropy, squeezed, for any finite $n$, between
the CM entropy~\cite{bender1978asymptotic,bianconi2009entropy,barvinok2013number} and the HSCM entropy~\eqref{eq:gibbs-bernoulli-graphon-entropy}.
In other words, dense CM, SCM, and HSCM are all equivalent in the limit,
versus the sparse case where the equivalence is broken~\cite{anand2009entropy,squartini2015breaking} since the graphon is zero in the limit.

The key point here however is that the rescaled degree distribution in the dense HSCM converges to a well-defined limit, i.e.,
satisfies the hypersoft constraints~\eqref{eq:dense-P-limit},
and that the HSCM Gibbs entropy converges to the graphon entropy. Therefore if we \emph{define} a maximum-entropy ensemble
under given hypersoft constraints $f$ to be an ensemble that: 1)~satisfies these constraints, i.e., has a degree distribution
converging to $f$ in the limit, 2)~maximizes graphon entropy under these constraints given by~\eqref{eq:dense_graphon_entropy_constraint},
and 3)~has rescaled Gibbs entropy converging to the graphon entropy, then the dense HSCM ensemble is trivially such an ensemble.
In addition, this ensemble is the unique maximum-entropy hypersoft ensemble in the dense case~\cite{chatterjee2011random}.

These observations instruct us to extend this definition of maximum-entropy hypersoft dense graphs to sparse graphs,
where we naturally replace the dense hypersoft constraints~\eqref{eq:dense-P-limit} with
sparse hypersoft constraints~(\ref{eq:convergence_degree_hscm_constraint},\ref{eq:convergence_expected_degree_constraint_hscm}).
However, things become immediately less trivial in this case. In particular, we face the difficulty that since the limit graphon of
any sparse exchangeable graph ensemble is zero according to
the Aldous-Hoover theorem~\cite{aldous1981representations,hoover1979relations}, the entropy of this graphon is zero as well. Since this entropy is zero, $|\mathcal{S}^\ast[G_n] - \sigma[G_n]| \to 0$
in~\eqref{eq:convergence_dense_gibbs_entropy} does not necessarily imply that the rescaled Gibbs entropy converges to the graphon entropy.
We address this difficulty next.

\subsubsection{Rescaled graphon entropy of sparse graphs}

Consider again the generalized graphon ensemble $\G_{W,\mu_n,n}$ of random graphs defined in Section~\ref{sec:graphon-ensembles}, and their graphon entropy defined as
\begin{equation}\label{eq:def_graphon_entropy_n}
	\sigma[G_n] = \iint_{A_n^2} H\left(W(x,y)\right)\,d\mu_n(x)\,d\mu_n(y).
\end{equation}
If $W\neq\{0,1\}$ everywhere, then for any finite $n$, $\sigma[G_n]$ is positive, but if the ensemble is sparse, then $\lim_{n\to\infty}\sigma[G_n]=0$.

To address this problem we rescale the graphon entropy $\sigma[G_n]$ such that upon rescaling it converges to a positive constant. That is, let $a_n$ be a sequence such that
\begin{equation}\label{eq:def_rescaled_graphon_entropy}
	\lim_{n \to \infty}  a_n \sigma[G_n] = a \in (0, \infty).
\end{equation}
This rescaling does not affect the graphon entropy maximization problem, 
because maximizing $\sigma[G_n]$ for every $n$ as a functional of $W$ under a given constraint is equivalent to maximizing $a_n\sigma[G_n]$ under the same constraint.

Upon this rescaling, we see that the rescaled Gibbs entropy converges to the graphon entropy, generalizing~\eqref{eq:convergence_dense_gibbs_entropy}, if
\begin{equation}\label{eq:def_maximal_entropy_graphon_ensemble}
	\lim_{n \to \infty} a_n\left|\mathcal{S}^\ast[G_n] - \sigma[G_n]\right| = 0,
\end{equation}
in which case $a_n \mathcal{S}^\ast[G_n]$ converges to $a$. This condition implies that the rescaled Gibbs entropy $\mathcal{S}^\ast[G_n]$ converges to the graphon entropy $\sigma[G_n]$ faster than either of them converge to zero.

\subsubsection{Sparse maximum-entropy graphs with a given degree distribution}\label{sec:entropy-maximization-definition}

With the graphon rescaling in the previous section, we can now \emph{define} a graphon ensemble $\mathcal{G}_{W, \mu_n, n}$ to be a maximum-entropy ensemble
under the sparse hypersoft constraints~(\ref{eq:convergence_degree_hscm_constraint},\ref{eq:convergence_expected_degree_constraint_hscm})
if:
\begin{enumerate}[\upshape 1)]
	\item the degree distribution and expected degree in $\mathcal{G}_{W, \mu_n, n}$ converge to~(\ref{eq:convergence_degree_hscm_constraint},\ref{eq:convergence_expected_degree_constraint_hscm});
	\item graphon entropy $\sigma[G_n]$ of $\mathcal{G}_{W, \mu_n, n}$ is maximized for every $n$ under the constraint imposed by~(\ref{eq:convergence_degree_hscm_constraint},\ref{eq:convergence_expected_degree_constraint_hscm}); and
	\item \eqref{eq:def_maximal_entropy_graphon_ensemble} holds, with $a_n$ given by \eqref{eq:def_rescaled_graphon_entropy}.
\end{enumerate}

Our main result (Theorem \ref{thm:main_result}) is that the sparse power-law hypersoft configuration model, defined next, is a maximum-entropy model under hypersoft constraints. 

\subsubsection{Sparse power-law hypersoft configuration model (sparse HSCM)}\label{sec:HSCM}

The sparse power-law \texttt{HSCM}$(\gamma,\nu)$ is defined as the graphon ensemble $\mathcal{G}_{W, \mu_n, n}=\mathcal{G}_{W, \mu, \mathbf{A}}$, Section~\ref{sec:graphon-ensembles}, with
\begin{align}
  W(x,y) &= \frac{1}{e^{x+y}+1}, \label{eq:def_hscm_graphon} \\
  \mu &= e^{\gamma x},\quad\gamma>1, \label{eq:def_hscm_mu} \\
  A_n &= (-\infty, R_n], \label{eq:def_hscm_A_n} \\
  R_n &= \frac{1}{2}\log\frac{n}{\nu\beta^2},\quad\nu>0,\quad\beta=1-\frac{1}{\gamma}, \label{eq:def_hscm_R_n} \\
  \mu_n &= \frac{\mu|_{A_n}}{\mu(A_n)}= \gamma\,e^{\gamma(x-R_n)}. \label{eq:def_hscm_mu_n}
\end{align}
The dense power-law HSCM is recovered from the above definition by setting $\nu=\tilde{\nu}n$, where $\tilde{\nu}$ is a constant, in which case $R_n=R=-(1/2)\log\left(\tilde{\nu}\beta^2\right)$, $A_n=A=(-\infty,R]$, and $\mu_n=\mu=\gamma\,e^{\gamma(x-R)}$.

\section{Results}\label{sec:results}

In this section we formally state our results, and provide brief overviews of their proofs appearing in subsequent sections. The main result is Theorem~\ref{thm:main_result}, stating that the \texttt{HSCM}$(\gamma,\nu)$ defined in Section~\ref{sec:HSCM} is a maximum-entropy model under hypersoft power-law degree distribution constraints, according to the definition in Section~\ref{sec:entropy-maximization-definition}.
This result follows from Theorems~\ref{thm:mixed_poisson_degrees_hscm}-\ref{thm:convergence_graph_entropy} and Proposition~\ref{prop:max_entropy_graphon}.
Theorems~\ref{thm:mixed_poisson_degrees_hscm},\ref{thm:average_degree_hscm} establish the limits of the degree distribution and expected average degree in the \texttt{HSCM}$(\gamma,\nu)$.
Proposition~\ref{prop:max_entropy_graphon} states that \texttt{HSCM}'s graphon uniquely maximizes the graphon entropy under the constraints imposed by the degree distribution.
Theorem~\ref{thm:convergence_graphon_entropy} establishes proper graphon rescaling and the limit of the rescaled graphon.
Finally, the most critical and involved Theorem~\ref{thm:convergence_graph_entropy} proves that the rescaled Gibbs entropy of the \texttt{HSCM} converges to its rescaled graphon entropy.

\subsection{Main result}

Let $Y$ be a Pareto random variable with shape $\gamma>1$, scale $\nu\beta>0$, $\beta=1-1/\gamma$, so that $Y$'s probability density function is
\begin{align}\label{eq:pareto-pdf}
  P_Y(y)&=\gamma\left(\nu\beta\right)^\gamma y^{-\alpha},\quad y\geq\nu\beta,\quad \alpha=\gamma+1,\quad\text{and}\\
  \label{eq:pareto_tail_cdf}
	\Prob{Y > y} &= \begin{cases}
		(\nu \beta)^\gamma y^{-\gamma} &\mbox{if } y \ge \nu \beta \\
		1 &\mbox{otherwise.} 
	\end{cases}
\end{align}
Let $\D$ be a discrete random variable with probability density function
\begin{equation}\label{eq:def_mixed_poisson_degrees}
	\Prob{\mathscr{D} = k} = \Exp{\frac{Y^k}{k!} e^{-Y}},\quad k=0,1,2,\ldots,
\end{equation}
which is the mixed Poisson distribution with mixing parameter $Y$~\cite{grandell1997mixed}. Then it follows that
\begin{equation}\label{eq:expected-degree}
\Exp{\mathscr{D}} = \Exp{Y} = \nu,
\end{equation}
and since $Y$ is a power law with exponent $\alpha$, the tail of $\mathscr{D}$'s distribution is also a power law with the same exponent~\cite{grandell1997mixed}.
In particular, $\Prob{\D = k}$ is given by \eqref{eq:P(k)}.
Therefore if $\D$ is the degree of a random node in a random graph ensemble, then graphs in this ensemble are sparse and have a power-law
degree distribution.

Our main result is:
\begin{theorem}\label{thm:main_result}
For any $\gamma>1$ and $\nu>0$, \texttt{HSCM}$(\gamma, \nu)$ is a maximum entropy ensemble of random graphs under the hypersoft constraints~(\ref{eq:convergence_degree_hscm_constraint},\ref{eq:convergence_expected_degree_constraint_hscm})
with $\Prob{\mathscr{D} = k}$ and $\nu$ defined by~(\ref{eq:pareto-pdf}-\ref{eq:expected-degree}).
\end{theorem}

\subsection{The limit of the degree distribution in the \texttt{HSCM}}

The degree $\D_n$ of a random node $i$ in a random \texttt{HSCM}$(\gamma, \nu)$ graph of size $n$, conditioned on the node coordinates
${\bf x}_n=x_1\ldots x_n$, is the sum of $n - 1$ independent Bernoulli random 
variables with success probabilities $W(x_i,x_j)$, $j\neq i$. The
distribution of this sum can be approximated by the mixed Poisson distribution with the mixing parameter $\sum_{j \neq i} W(x_i, x_j)$.
Therefore after first integrating over $x_j$ with $j \ne i$ and then over $x_i$, the distribution of $\D_n$ is approximately the mixed
Poisson distribution 
\begin{align*}
	\Prob{{\D}_n = k} = \Exp{\frac{(\kappa_n(X))^k}{k!}\,e^{-\kappa_n(X)}}, \quad \text{so that $\Exp{{\D}_n} = \Exp{\kappa_n(X)}$},
\end{align*}
where the random variable $X$ has density $\mu_n$~\eqref{eq:def_hscm_mu_n}, and the mixing parameter $\kappa_n(x)$ is the expected degree of a node at 
coordinate $x$:
\begin{align}
    \kappa_n(x) &= (n-1)w_n(x),\\
	w_n(x) &= \int_{A_n} W(x,y) \, d\mu_n(y),\label{eq:def_expected_degree_function_hscm}
\end{align}
where $A_n$ is given by~\eqref{eq:def_hscm_A_n}.

In the $n\to\infty$ limit, the distribution of the expected degree $\kappa_n(X)$ of $X$ converges to the Pareto distribution with shape $\gamma$ and scale $\nu\beta$.
To prove this, we use the observation that the mass of measure $\mu_n$ is concentrated towards the right end of the interval $A_n=(-\infty, R_n]$, where
$R_n\gg 1$ for large $n$. Therefore, not only the contributions coming from negative $x,y$ are negligible, but we can also
approximate the Fermi-Dirac graphon $W(x,y)$ in~\eqref{eq:def_hscm_graphon} with its classical limit approximation
\begin{equation}\label{eq:calssical_limit_approximation}
  \widehat{W}(x,y) = e^{-(x + y)}
\end{equation}
on $\R_+^2$. In addition, the expected degree function $w_n(x)$ can be approximated with $\widehat{w}_n(x)$ defined by
\begin{align}\label{eq:def_expected_degree_function_approximation_hscm}
	\widehat{w}_n(x) &= \begin{cases}
		\omega_n e^{-x} &\mbox{if } 0 \le x \le R_n\\
		0 &\mbox{otherwise}
	\end{cases}, 
	\text{ where }\\
    \omega_n &= \int_{0}^{R_n} e^{-x} \, d\mu_n(x) = \frac{1-e^{-(\gamma-1) R_n}}{\beta e^{R_n}} =\left(\frac{\nu}{n}\right)^\frac{1}{2}+o\left(n^{-\frac{1}{2}}\right),
\end{align}
so that the expected degree of a node at coordinate $x$ can be approximated by
\begin{align*}
  \widehat{\kappa}_n(x) = n \widehat{w}_n(x) = e^{-x}\left(\left(\nu n\right)^\frac{1}{2}+o\left(n^{\frac{1}{2}}\right)\right).
\end{align*}

To see that $\widehat{\kappa}_n(X)$ converges to a Pareto random variable, note that since $X$ has density $\mu_n$, it follows that for all $t > \nu\beta$
\begin{align*}
	\Prob{\widehat{\kappa}_n(X) > t} &= \Prob{X < \log\frac{n \omega_n}{t}} = e^{-\gamma R_n} \left(\frac{n \omega_n}{t}\right)^{\gamma}\\
	&= \left(\nu\beta\right)^\gamma t^{-\gamma}\left(1 + o(1)\right) = \Prob{Y > t}(1 + o(1)),
\end{align*}
where $Y$ is a Pareto-distributed random variable~\eqref{eq:pareto_tail_cdf}.
We therefore have the following result, the full proof of which can be found 
in Section \ref{sssec:proof_degree_distribution_hscm}:
\begin{theorem}[\texttt{HSCM}$(\gamma, \nu)$ satisfies \eqref{eq:convergence_degree_hscm_constraint}]\label{thm:mixed_poisson_degrees_hscm}
Let $\gamma > 1$, $\nu > 0$, and ${\D}_n$ be the degree of a uniformly chosen vertex in the \texttt{HSCM}$(\gamma, \nu)$ graphs of size $n$.
Then, for each $k=0,1,2,\ldots$,
\begin{align*}
	\lim_{n \to \infty} \Prob{\mathscr{D}_n = k} = \Prob{\mathscr{D} = k},
\end{align*}
where $\Prob{\mathscr{D} = k}$ is given by~\eqref{eq:def_mixed_poisson_degrees}.
\end{theorem}

\subsection{The limit of the expected average degree in the \texttt{HSCM}}

The expected degree of a random node in $n$-sized \texttt{HSCM} graphs is, for any fixed $i$,
\begin{equation}\label{eq:average_degree_hscm}
	\Exp{\mathscr{D}_n} = \sum_{j \neq i} \Exp{W(X_i, X_j)} = (n - 1)\Exp{W(X,Y)},
\end{equation}
where $X$ and $Y$ are independent random variables with distribution $\mu_n$. Approximating $W(x, y)$ with $\widehat{W}(x,y)$ on $\R_+^2$, and using $e^{2R_n} = n/\nu\beta^2$,
we have
\begin{align*}
	n \iint_{0}^{R_n} \widehat{W}(x,y)\, d\mu_n(x)\,d\mu_n(y)
	= n\left(\int_{0}^{R_n} e^{-x} \, d\mu_n(x)\right)^2 
	= \frac{n}{\beta^2e^{2R_n}} \left(1 - e^{-(\gamma - 1)R_n}\right)^2 = \nu + o(1).
\end{align*}

All other contributions are shown to also vanish in the $n\to\infty$ limit in Section \ref{sssec:proof_average_degree_hscm}, where the following theorem is proved:
\begin{theorem}[\texttt{HSCM} satisfies \eqref{eq:convergence_expected_degree_constraint_hscm}]\label{thm:average_degree_hscm}
Let $\gamma > 1$, $\nu > 0$, and ${\D}_n$ be the degree of a uniformly chosen vertex in the \texttt{HSCM}$(\gamma, \nu)$ graphs of size $n$.
Then
\begin{align*}
	\lim_{n \to \infty} \Exp{\D_n} = \nu.
\end{align*}
\end{theorem}

\subsection{\texttt{HSCM} maximizes graphon entropy}

Let $A \subseteq \R$ be some interval, $\mu$ a measure on $A$ with $\mu(A) < \infty$, and suppose some $\mu$-integrable function $w : A \to \R$ is given.
Consider the graphon entropy maximization problem under the constraint
\begin{equation}\label{eq:def_graphon_entropy_constraint}
	w(x) = \int_A W(x,y) \,d\mu(y).
\end{equation}
That is, the problem is to find a symmetric function $W^*$ that maximizes graphon entropy $\sigma[W,\mu,A]$ in~\eqref{eq:def_graphon_entropy}
and satisfies the constraint above, for fixed $A$ and $\mu$.

We note that this problem is a ``continuous version'' of the Gibbs entropy maximization problem in the SCM
ensemble in Section~\ref{sec:SCM}. The following proposition,
which we prove in Section \ref{sssec:proof_max_entropy_graphon}, states that the solution to this problem
is a ``continuous version'' of the SCM solution~\eqref{eq:def_scm_lambda}:
\begin{proposition}[\texttt{HSCM} Maximizes Graphon Entropy]\label{prop:max_entropy_graphon}
Suppose there exists a solution $W^*$ to the graphon entropy maximization problem defined above.
Then this solution has the following form:
\begin{equation}\label{eq:def_max_entropy_graphon_solution}
	W^*(x,y) = \frac{1}{e^{\lambda(x) + \lambda(y)} + 1},
\end{equation}
where function $\lambda$ is $\mu$-almost-everywhere uniquely defined on $A$ by~\eqref{eq:def_graphon_entropy_constraint}.
\end{proposition}

This proposition proves that for each $n$, the~\texttt{HSCM}$(\gamma,\nu)$'s Fermi-Dirac graphon~\eqref{eq:def_hscm_graphon}
maximizes the graphon entropy under constraint~\eqref{eq:def_expected_degree_function_hscm},
because $\mathbf{A}$ and $\mu$ in the \texttt{HSCM}$(\gamma,\nu)$ are chosen such that $\lambda(x)=x$.
This is always possible as soon as $\lambda(x)$ is invertible, cf.~Section~\ref{sec:dense_HSCM}.
For each $n$, interval $A_n$~\eqref{eq:def_hscm_A_n} and measure $\mu_n$~\eqref{eq:def_hscm_mu_n} in the \texttt{HSCM}$(\gamma,\nu)$
can be mapped to $[0,1]$ and $1$, respectively, in which case $\lambda_n(x)=R_n+(1/\gamma)\log x$, leading to~\eqref{eq:W_n}.
In other words, node coordinates $x$ in the original \texttt{HSCM}$(\gamma,\nu)$ definition in Section~\ref{sec:HSCM}, 
and their coordinates $\tilde{x}$ in its equivalent definition with $A_n=[0,1]$ and $\mu_n=1$ are related by
\begin{align*}
  \tilde{x}=e^{\gamma(x-R_n)}.
\end{align*}

\subsection{Graphon entropy scaling and convergence}

To derive the rate of convergence of $\texttt{HSCM}$'s graphon entropy to zero, it suffices
to consider the classical limit approximation $\widehat{W}$~\eqref{eq:calssical_limit_approximation} to~$W$~\eqref{eq:def_hscm_graphon} on $\R_+^2$.
Its Bernoulli entropy~\eqref{eq:def_p_entropy} is
\begin{align*}
	H\left(\widehat{W}(x,y)\right) = (x + y)\widehat{W}(x,y) - (1 - \widehat{W}(x,y))\log(1 - \widehat{W}(x,y)).
\end{align*}
Since the most of the mass of $\mu_n$ is concentrated near $R_n$ and $\widehat{W}(R_n, R_n) \to 0$, the second term
is negligible. Integrating the first term over $[0, R_n]$, we get 
\begin{align*}
	&\iint_0^{R_n} (x + y) \widehat{W}(x,y) \, d\mu_n(x)\,d\mu_n(y)
	= \gamma^2  e^{-2\gamma R_n} \iint_0^{R_n} (x + y) e^{(\gamma - 1)(x + y)} \, dx\,dy \\
	&= \frac{2 R_n }{\beta^2 e^{2R_n}} + O(n^{-1})
	= \frac{\nu}{n}\log\frac{n}{\nu\beta^2} + O(n^{-1})
	= \nu\,\frac{\log n}{n} + O(n^{-1}), \numberthis \label{eq:convergence_approx_graphon_entropy}
\end{align*}
from which we obtain the proper scaling as $\log(n)/n$. All further details behind the proof of the following theorem are in Section~\ref{sssec:proof_convergence_graphon_entropy}.
\begin{theorem}[Graphon Entropy Convergence]\label{thm:convergence_graphon_entropy}
Let $\sigma[G_n]$ be the graphon entropy~\eqref{eq:def_graphon_entropy_n} in the \texttt{HSCM}$(\gamma,\nu)$ ensemble with any $\gamma > 1$ and $\nu>0$.
Then, as $n \to \infty$,
\begin{align*}
	\left|\frac{n\,\sigma[G_n]}{\log n}- \nu\right| = O\left(1/\log n\right).
\end{align*}
\end{theorem}
This theorem implies that $\sigma[G_n]$ goes to zero as $\log(n)/n$, while $n\,\sigma[G_n]/\log n$
goes to $\nu$ as $1/\log n$. 

\subsection{Gibbs entropy scaling and convergence}\label{ssec:convergence_graph_entropy}

The last part of Theorem~\ref{thm:main_result} is to prove that the rescaled Gibbs entropy~\eqref{eq:rescaled_gibbs_entropy} of the \texttt{HSCM} converges to
its graphon entropy~\eqref{eq:def_graphon_entropy_n} faster than the latter converges to zero.

The graphon entropy is a trivial lower bound for the rescaled Gibbs entropy, Section~\ref{sec:bernoulli_graphon_entropies},
so the problem is to find an appropriate upper bound for the latter converging to the graphon entropy. To identify such an upper bound,
we rely on an argument similar to~\cite{janson2013graphons}. Specifically, we first partition $A_n$ into $m$ intervals $I_{t}$ 
that induce a partition of $A_n^2$ into rectangles $I_{st}=I_s\times I_t$, $s,t=1\ldots m$.
We then approximate the graphon by its average value on each rectangle. Such approximation brings in an error term on each rectangle. We then show that the Gibbs entropy
is upper-bounded by the entropy of the averaged graphon, plus the sum of entropies of indicator random variables $M_i$ which take value $t$ if the coordinate
$x_i$ of node $i$ happens to fall within interval $I_t$. The smaller the number of intervals $m$, the smaller the total entropy of these random variables $M_i$,
but the larger the sum of the error terms coming from graphon averaging, because rectangles $I_{st}$ are large. The smaller they are, the smaller the total error term,
but the larger the total entropy of the $M_i$'s. The crux of the proof is to find a ``sweet spot''---the right number of intervals guaranteeing the proper
balance between these two types of contributions to the upper bound, which we want to be tighter than the rate of the convergence of the graphon entropy to zero.

This program is executed in Section~\ref{ssec:proof_convergence_graph_entropy}, where we prove the following theorem:
\begin{theorem}[Gibbs Entropy Convergence]\label{thm:convergence_graph_entropy}
Let $\sigma[G_n]$ be the graphon entropy~\eqref{eq:def_graphon_entropy_n} and $\mathcal{S}^\ast[G_n]$ be the rescaled Gibbs entropy~\eqref{eq:rescaled_gibbs_entropy}
in the \texttt{HSCM}$(\gamma,\nu)$ ensemble with any $\gamma > 1$ and $\nu>0$.
Then
\begin{align*}
&	\lim_{n \to \infty} \frac{n}{\log n}\left|\mathcal{S}^\ast[G_n] - \sigma[G_n]\right| = 0, \text{ and}\\
&	\lim_{n \to \infty} \frac{2\,\mathcal{S}[G_n]}{n \log n } = \nu.
\end{align*}
\end{theorem}
We remark that this theorem implies that 
\begin{align*}
  \mathcal{S}[G_n] \sim \frac{\nu}{2}\,n \log n,
\end{align*}	
which is the leading term of the (S)CM Gibbs entropy obtained in~\cite{anand2014entropy}. 
It is also instructive to compare this scaling of Gibbs entropy with its scaling in dense ensembles
with $\lim_{n\to\infty}\sigma[G_n]=\sigma[W,\mu,A]=\sigma\in(0,\infty)$~\cite{janson2013graphons}:
\begin{align*}
	\mathcal{S}[G_n] \sim \frac{\sigma}{2}\,n^2.
\end{align*}	

Finally, it is worth mentioning that even though we use the Fermi-Dirac graphon $W$~\eqref{eq:def_hscm_graphon} to define our $W$-random graphs,
the same convergence results could be obtained for $W$-random graphs defined by any other graphon
$W^\prime$ such that
\begin{align*}
	\lim_{n \to \infty} n \, \Exp{\left|W(X,Y)- W^\prime(X,Y)\right|} = 0,
\end{align*}
with $X$ and $Y$ having density $\mu_n$. In fact, to establish the required limits, we use the classical limit approximation graphon $\widehat{W}$ instead of $W$.
Therefore there exists a vast equivalence class of $W$-random graphs defined by graphons $W^\prime$ that all have the same limit degree distribution~\eqref{eq:convergence_degree_hscm_constraint}
and average degree~\eqref{eq:convergence_expected_degree_constraint_hscm}, and whose rescaled Gibbs entropy converges to the graphon entropy of the Fermi-Dirac~$W$.
However, it follows from Proposition~\ref{prop:max_entropy_graphon} that among all these ensembles, only the $W$-random graph ensemble defined by the Fermi-Dirac
graphon~\eqref{eq:def_hscm_graphon} uniquely maximizes the graphon entropy~\eqref{eq:def_graphon_entropy} for each $n$, which, by our definition of maximum-entropy
ensembles under hypersoft constraints, is a necessary condition for graph entropy maximization.

\section{Proofs}\label{sec:proofs}

In this section we provide the proofs of all the results stated in the previous section. In Section \ref{ssec:graphon_approximations} we begin with some preliminary results on the accuracy of the approximation of the Fermi-Dirac graphon $W$~\eqref{eq:def_hscm_graphon} by the classical limit approximation $\widehat{W}$. In the same section we also establish results showing that the main contribution of the integration with respect to $\mu_n$ is on the positive part $[0, R_n]$ of the interval $A_n$ as defined by \eqref{eq:def_hscm_A_n}. In particular, we show that all contributions coming from the negative part of this interval, i.e., $\R_-$, are $o(n^{-1})$, which means that for all our results the negative part of the support of our measure $\mu_n$ is negligible. 
We then proceed with proving Theorems \ref{thm:mixed_poisson_degrees_hscm}
and \ref{thm:average_degree_hscm} in Section~\ref{ssec:proofs_degrees_hscm}. The proofs of Proposition \ref{prop:max_entropy_graphon} and Theorem \ref{thm:convergence_graphon_entropy}
can be found in Section~\ref{ssec:proof_graphon_entropy}. Finally, the convergence of the rescaled Gibbs entropy to the graphon entropy (Theorem \ref{thm:convergence_graph_entropy}) is given in Section~\ref{ssec:proof_convergence_graph_entropy}.

\subsection{The classical limit approximation of the Fermi-Dirac graphon}\label{ssec:graphon_approximations}

We will use $e^{-(x + y)}$ as an approximation to the graphon $W$ to compute all necessary limits. To be precise we define
\begin{equation}\label{eq:def_approx_graphon}
	\widehat{W}(x,y) = \min\{e^{-(x + y)},1\}
\end{equation}
and show that differences between the integrals of $W$ and $\widehat{W}$ converge to zero as $n$ tends to infinity. 
Note that, instead of $\widehat{W}$, we could have also worked with the integral expressions involving $W$, which might have led to 
better bounds. However, these integrals tend to evaluate to combinations of hypergeometric functions, while the integrals of 
$\widehat{W}$ are much easier to evaluate and are sufficient for our purposes. 

By the definition of $\widehat{W}(x,y)$ we need to consider separately, the intervals $(-\infty, 0]$ and $(0, R_n]$.
Since graphons are symmetric functions, this leads to the following three different cases:
\begin{equation}\label{eq:def_spliting_A_n}
	\textrm{I)} \quad -\infty < x, y \le 0 \qquad \textrm{II)} \quad -\infty < y \le R_n \text{ and } 0 < x \le R_n, \qquad \textrm{III)} \quad 0 < x, y \le R_n.
\end{equation}
For case I) we note that $W, \, \widehat{W} \le 1$ and 
\begin{equation}\label{eq:mu_n_R_n_0}
	\iint_{-\infty}^0 \, d\mu_n(y) \, d\mu_n(x) = O\left(n^{-\gamma}\right).
\end{equation}
With this we obtain the following result, which shows that for both $W$ and $\widehat{W}$, only the
integration over $(0, R_n]^2$, i.e.\ case III, matters.

\begin{lemma}\label{lem:integration_W_W_hat}
\begin{align*}
	\iint_{-\infty}^{R_n} \widehat{W}(x,y) \, d\mu_n(y) \, d\mu_n(x)
	- \iint_{0}^{R_n} \widehat{W}(x,y) \, d\mu_n(y) \, d\mu_n(x) = O\left(n^{-\frac{\gamma + 1}{2}}\right)
\end{align*}
and the same result holds if we replace $\widehat{W}$ with $W$.
\end{lemma}

\begin{proof}
First note that
\begin{equation}\label{eq:integration_hat_W_negative}
	\int_{-\infty}^{-R_n} \int_{-\infty}^{R_n} \widehat{W}(x,y) \, d\mu_n(y) \, d\mu_n(x)
	\le \int_{-\infty}^{-R_n} d\mu_n(y) = O\left(n^{-\gamma}\right).
\end{equation}
We show that
\begin{equation}\label{eq:integration_hat_W_mixed}
	\int_{-R_n}^0 \int_{0}^{R_n} \widehat{W}(x,y) \, d\mu_n(y) \, d\mu_n(x) = O\left(n^{-\frac{\gamma + 1}{2}}\right),
\end{equation}
which together with \eqref{eq:mu_n_R_n_0} and \eqref{eq:integration_hat_W_negative} implies the first result. The result for $W$ follows by noting that $W \le \widehat{W}$.

We split the integral \eqref{eq:integration_hat_W_mixed} as follows
\begin{align*}
	&\hspace{-30pt}\int_{-R_n}^0 \int_{0}^{R_n} \widehat{W}(x,y) \, d\mu_n(y) \, d\mu_n(x) \\
	&= \int_{-R_n}^0 \int_{0}^{-x} \, d\mu_n(y) \, d\mu_n(x)
		+ \int_{-R_n}^0 \int_{-x}^{R_n} e^{-(x + y)} \, d\mu_n(y) \, d\mu_n(x).
\end{align*}
For the first integral we compute
\begin{align*}
	\int_{-R_n}^0 \int_{0}^{-x} \, d\mu_n(y) \, d\mu_n(x) &= \gamma e^{-2\gamma R_n} 
		\int_{-R_n}^0 \left(e^{-\gamma x} - 1\right)e^{\gamma x} \, dx \\
	&= e^{-2\gamma R_n} \left(\gamma R_n + e^{-\gamma R_n} - 1\right) = O\left(\log(n) n^{-\gamma}\right)
\end{align*}
Finally, the second integral evaluates to
\begin{align*}
	&\hspace{-60pt}\frac{\gamma}{\beta}e^{-2\gamma R_n}\int_{-R_n}^0 \left(e^{(\gamma - 1)R_n} - e^{-(\gamma - 1)x}\right)e^{(\gamma - 1)x} \, dx \\
	&=\frac{\gamma}{\beta}e^{-2\gamma R_n}\left(\frac{e^{(\gamma - 1)R_n} - 1}{\gamma - 1} - R_n\right)
	= O\left(n^{-\frac{\gamma + 1}{2}}\right),
\end{align*}
from which the result follows since $\gamma > (\gamma + 1)/2$.
\end{proof}

We can show a similar result for $H(W)$ and $H(\widehat{W})$.

\begin{lemma}\label{lem:integration_H_W_W_hat}
\begin{align*}
	\iint_{-\infty}^{R_n} H\left(\widehat{W}(x,y)\right) \, d\mu_n(y) \, d\mu_n(x)
	- \iint_{0}^{R_n} H\left(\widehat{W}(x,y)\right) \, d\mu_n(y) \, d\mu_n(x) = \epsilon_n,
\end{align*}
where
\begin{align*}
	\epsilon_n = \begin{cases}
			O\left(\log(n) n^{-\gamma}\right) &\mbox{if } 1 < \gamma < 2,\\
			O\left(\log(n)^2 n^{-2}\right) &\mbox{if } \gamma = 2,\\
			O\left(n^{-\frac{2 + \gamma}{2}}\right) &\mbox{if } \gamma > 2,
		\end{cases}
\end{align*}
Moreover, the same result holds if we replace $\widehat{W}$ with $W$.
\end{lemma}

\begin{proof}
We will first prove the result for $\widehat{W}$. For this we split the interval $A_n$ into three parts $(-\infty, -R_n]$, $[-R_n, 0]$ and $(0, R_n]$ 
and show that the integrals on all ranges other than $[0, R_n]^2$ are bounded by a term that scales as $\epsilon_n$. 

Since $H(p) \le \log(2)$ for all $0 \le p \le 1$ it follows from \eqref{eq:mu_n_R_n_0} and \eqref{eq:integration_hat_W_negative} that, for all $\gamma > 1$,
\begin{align*}
	&\iint_{-\infty}^0 H\left(\widehat{W}(x,y)\right) \, d\mu_n(y) \, d\mu_n(x) = O\left(n^{-\gamma}\right)
	= o(\epsilon_n), \\
	&\int_{-\infty}^{R_n} \int_{-\infty}^{-R_n} H\left(\widehat{W}(x,y)\right) \, d\mu_n(y) \, d\mu_n(x) = O\left(n^{-\gamma}\right)
		= o(\epsilon_n).
\end{align*}
Hence, using the symmetry of $\widehat{W}$, we only need to consider the integration over $(0, R_n] \times (-R_n, 0]$.

First we compute
\begin{align*}
	H\left(\widehat{W}(x,y)\right) 
	= e^{-(x+y)}(x + y) - (1 - e^{-(x + y)})\log\left(1 - e^{-(x + y)}\right)
\end{align*}
and observe that
\begin{equation}\label{eq:upper_bound_H_W_hat}
	- (1 - e^{-z})\log\left(1 - e^{-z}\right) \le e^{-2z} \quad \text{for all large enough $z$.}
\end{equation}
Now let $\delta > 0$ be such that \eqref{eq:upper_bound_H_W_hat} holds for all $z \ge \delta$. Then $\delta < R_n$ for sufficiently 
large $n$ and we split the integration as follows
\begin{align*}
	&\hspace{-30pt}\int_0^{R_n} \int_{-R_n}^0 H\left(\widehat{W}(x,y)\right) \, d\mu_n(y) \, d\mu_n(x)\\
	&= \int_0^{\delta} \int_{-R_n}^0 H\left(\widehat{W}(x,y)\right) \, d\mu_n(y) \, d\mu_n(x)
		+ \int_\delta^{R_n} \int_{-R_n}^0 H\left(\widehat{W}(x,y)\right) \, d\mu_n(y) \, d\mu_n(x). 
\end{align*}
The first integral is $O(n^{-\gamma})$. For the second we note that $x + y > \delta$ for all $y > \delta - x$ and hence
\begin{align*}
	&\hspace{-30pt}\int_\delta^{R_n} \int_{-R_n}^0 H\left(\widehat{W}(x,y)\right) \, d\mu_n(y) \, d\mu_n(x)\\
	&\le \int_\delta^{R_n} \int_{-R_n}^{\delta - x} \log(2) \, d\mu_n(y) \, d\mu_n(x)
		+ \int_\delta^{R_n} \int_{\delta - x}^0 e^{-2(x + y)} \, d\mu_n(y) \, d\mu_n(x). 
\end{align*}
For the second integral we obtain, 
\begin{align*}
	\int_{\delta}^{R_n} \int_{\delta - x}^0 e^{-2(x + y)} \, d\mu_n(y) \, d\mu_n(x) = \begin{cases}
		O\left(\log(n) n^{-\gamma}\right) &\mbox{if } 1 < \gamma < 2,\\
		O\left(\log(n)^2 n^{-2}\right) &\mbox{if } \gamma = 2,\\
		O\left(n^{-\frac{2 + \gamma}{2}}\right) &\mbox{if } \gamma > 2,
	\end{cases}
\end{align*}
while for the first integral we have
\begin{align*}
	\int_\delta^{R_n} \int_{-R_n}^{\delta - x} \log(2) \, d\mu_n(y) d\mu_n(x) 
	&\le \gamma e^{-2\gamma R_n}
		\int_{\delta}^{R_n} e^{\gamma \delta} \, dx = O\left(\log(n) n^{-\gamma}\right).
\end{align*}
Therefore we conclude that
\begin{align*}
	\int_0^{R_n} \int_{-\infty}^0 H\left(\widehat{W}(x,y)\right) \, d\mu_n(y) d\mu_n(x)
	= \begin{cases}
			O\left(\log(n) n^{-\gamma}\right) &\mbox{if } 1 < \gamma < 2,\\
			O\left(\log(n) n^{-2}\right) &\mbox{if } \gamma = 2,\\
			O\left(n^{-\frac{4 + \gamma}{2}}\right) &\mbox{if } \gamma > 2,
	\end{cases}
\end{align*}
which yields the result for $\widehat{W}$.

For $W$ we first compute that 
\begin{align*}
	H(W(x,y)) &= \log\left(1 + e^{x + y}\right) - \frac{(x + y) e^{x + y}}{1 + e^{x + y}}\\
		&= \log\left(1 + e^{x + y}\right) - (x + y) + (x + y)\left(1 - \frac{e^{x + y}}{1 + e^{x + y}}\right)\\
		&= \log\left(1 + e^{-(x + y)}\right) + \frac{(x + y)}{1 + e^{x + y}}
		\le \log\left(1 + e^{-(x + y)}\right) + (x + y)e^{-(x + y)}.
\end{align*}
Comparing this upper bound to $H(\widehat{W}(x,y))$ and noting that $\log\left(1 + e^{-z}\right)
\le e^{-2z}$ for large enough $z$, the result follows from the computation done for $\widehat{W}$.
\end{proof}

With these two lemmas we now establish two important results on the approximations of $W$. The first shows 
that if $X$ and $Y$ are independent with distribution $\mu_n$, then $\widehat{W}(X,Y)$ converges in expectation to $W(X,Y)$,
faster than $n^{-1}$.

\begin{proposition}\label{prop:concentration_W_W_hat}
Let $X, Y$ be independent with density $\mu_n$ and $\gamma > 1$. Then, as $n \to \infty$,
\begin{align*}
	\Exp{\left|W(X, Y) - \widehat{W}(X, Y)\right|} = O\left(n^{-\frac{\gamma + 1}{2}}\right)
\end{align*}
\end{proposition}

\begin{proof}
Since
\begin{align*}
	\Exp{\left|W(X, Y) - \widehat{W}(X, Y)\right|}
	= \iint_{-\infty}^{R_n} \left|W(x, y) - \widehat{W}(x, y)\right| \, d\mu_n(y) \, d\mu_n(x),
\end{align*}
and $\left|W(x, y) - \widehat{W}(x, y)\right| \le 1$ it follows from Lemma \ref{lem:integration_W_W_hat} that it is 
enough to consider the integral
\begin{align*}
	\iint_0^{R_n} \left|W(x, y) - \widehat{W}(x, y)\right| \, d\mu_n(y) \, d\mu_n(x).
\end{align*}
For this we note that
\begin{align*}
	\left|W(x,y) - \widehat{W}(x,y)\right| \le e^{-2(x + y)}.
\end{align*}
Hence we obtain
\begin{align*}
	&\hspace{-30pt}\iint_0^{R_n} \left|W(x, y) - \widehat{W}(x, y)\right| \, d\mu_n(y) \, d\mu_n(x)\\
	&\le \left(\int_0^{R_n} e^{-2x} \, d\mu_n(x)\right)^2
		= \begin{cases}
			O\left(n^{-\gamma}\right) &\mbox{if } 1 < \gamma < 2\\
			O\left(\log(n)^2 n^{-2}\right) &\mbox{if } \gamma = 2\\
			O\left(n^{-2}\right) &\mbox{if } \gamma > 2.
		\end{cases} 
\end{align*}
Since all these terms are $o(n^{-(\gamma + 1)/2})$, the result follows.
\end{proof}

Next we show that also $H(\widehat{W}(X, Y))$ converges in expectation to $H(W(X,Y))$, faster than $n^{-1}$.

\begin{proposition}\label{prop:concentration_Gibbs_W_W_hat}
Let $X, Y$ be independent with density $\mu_n$ and $\gamma > 1$. Then, as $n \to \infty$,
\begin{align*}
	\Exp{\left|H\left(W(X, Y)\right) - H\left(\widehat{W}(X, Y)\right)\right|} 
	= \begin{cases}
			O\left(\log(n) n^{-\gamma}\right) &\mbox{if } 1 < \gamma < 2\\
			O\left(\log(n)^3 n^{-2}\right) &\mbox{if } \gamma = 2\\
			O\left(\log(n) n^{-2}\right) &\mbox{if } \gamma > 2.
		\end{cases}
\end{align*}
\end{proposition}

\begin{proof}
Similar to the previous proof we now use Lemma \ref{lem:integration_H_W_W_hat} to show that it is enough
to consider the integral
\begin{align*}
	\iint_0^{R_n} \left|H\left(W(x,y)\right) - H\left(\widehat{W}(x,y)\right)\right| \, d\mu_n(y) d\mu_n(x).
\end{align*}

Define 
\begin{align*}
	\delta_W(x,y) = \frac{1}{e^{2(x+y)} + e^{x + y}},
\end{align*}
and note that $\widehat{W}(x,y) = W(x,y) + \delta_W(x,y)$. Now fix $x, y$. Then, by the Taylor-Lagrange theorem, there
exists a $W(x,y) \le c_W(x,y) \le \widehat{W}(x,y)$ such that
\begin{align*}
	\left|H\left(W(x,y)\right) - H\left(\widehat{W}(x,y)\right)\right| &= \left|H^\prime\left(c_W(x,y)\right)\right| \delta_W(x,y)\\
	&\le \left(\left|H^\prime\left(W(x,y)\right)\right| + \left|H^\prime\left(\widehat{W}(x,y)\right)\right|\right)\delta_W(x,y).
\end{align*}
Next we compute that
\begin{align*}
	\left|H^\prime\left(W(x,y)\right)\right| = (x + y), \quad \text{and} \quad
	\left|H^\prime\left(\widehat{W}(x,y)\right)\right| = \log\left(e^{x + y} - 1\right),
\end{align*}
where $\log(e^{x + y} - 1) \le (x + y)$ for all $x + y \ge 1$. We now split the integral and bound it as follows
\begin{align*}
	&\hspace{-30pt}\iint_0^{R_n} \left|H\left(W(x,y)\right) - H\left(\widehat{W}(x,y)\right)\right| \, d\mu_n(y) \, d\mu_n(x)\\
	&= \int_0^1 \int_0^{1-x} \left|H\left(W(x,y)\right) - H\left(\widehat{W}(x,y)\right)\right| \, d\mu_n(y) \, d\mu_n(x) \\
	&\hspace{10pt}+ \int_0^1 \int_{1-x}^{R_n} \left|H\left(W(x,y)\right) - H\left(\widehat{W}(x,y)\right)\right| \delta_W(x,y) \, d\mu_n(y) \, d\mu_n(x)\\
	&\hspace{10pt}+ \int_1^{R_n} \int_0^{R_n} \left|H\left(W(x,y)\right) - H\left(\widehat{W}(x,y)\right)\right| \delta_W(x,y) \, d\mu_n(y) \, d\mu_n(x)\\
	&\le \log(2) \int_0^1 \int_0^{1-x} \, d\mu_n(y) \, d\mu_n(x)\\
	&\hspace{10pt}+ 2\int_0^1 \int_{1-x}^{R_n} (x + y) \delta_W(x,y) \, d\mu_n(y) \, d\mu_n(x) \\
	&\hspace{10pt}+ 2\int_1^{R_n} \int_0^{R_n} (x + y) \delta_W(x,y) \, d\mu_n(y) \, d\mu_n(x) \\
	&\le \log(2) \int_0^1 \int_0^{1-x} \, d\mu_n(y) \, d\mu_n(x) + 4 \iint_0^{R_n} (x + y) \delta_W(x,y) \, d\mu_n(y) \, d\mu_n(x).
\end{align*}
The first integral is $O(\log(n) n^{-\gamma})$, while for the second we have
\begin{align*}
	&\hspace{-30pt}4 \iint_0^{R_n} (x + y) \delta_W(x,y) \, d\mu_n(y) \, d\mu_n(x)\\
	&\le 8 R_n \left(\int_0^{R_n} e^{-2x} \, d\mu_n(x)\right)^2
			= \begin{cases}
				O\left(\log(n) n^{-\gamma}\right) &\mbox{if } 1 < \gamma < 2\\
				O\left(\log(n)^3 n^{-2}\right) &\mbox{if } \gamma = 2\\
				O\left(\log(n) n^{-2}\right) &\mbox{if } \gamma > 2.
			\end{cases} 
\end{align*}
Comparing these scaling to the ones from Lemma \ref{lem:integration_H_W_W_hat} we see that the former are dominating, which finishes the proof.
\end{proof}

\subsection{Proofs for node degrees in the \texttt{HSCM}}\label{ssec:proofs_degrees_hscm}

In this section we give the proofs of Theorem \ref{thm:mixed_poisson_degrees_hscm} and Theorem \ref{thm:average_degree_hscm}. 
Denote by $D_i$ the degree of node $i$ and recall that $\D_n$ is the degree of a node, sampled uniformly at random.
Since the node labels are interchangeable we can, without loss of generality, consider $D_1$ for $\D_n$. 

For Theorem \ref{thm:average_degree_hscm} we use \eqref{eq:average_degree_hscm}. We show that if $X$ and $Y$ are independent 
with distribution $\mu_n$, then $\Exp{\widehat{W}(X,Y)} \to \nu$. The final result will then follow from Proposition 
\ref{prop:concentration_W_W_hat}.

The proof of Theorem \ref{thm:mixed_poisson_degrees_hscm} is more involved. Given the coordinates $X_1, \dots X_n$ we have 
$\mathscr{D}_n = \sum_{j = 2}^n W(X_1, X_j)$. We follow the strategy from \cite[Theorem 6.7]{hofstad2016random} and construct 
a coupling between $\mathscr{D}_n$ and a mixed Poisson random variable $P_n$, with mixing parameter $n \widehat{w}_n(X)$, where 
$\widehat{w}_n(x)$ is given by \eqref{eq:def_expected_degree_function_approximation_hscm} and $X$ has distribution $\mu_n$. 

In general, a coupling between two random variables $X$ and $Y$ consists of a pair of new random variables $(\widehat{X}, \widehat{Y})$,
with some joint probability distribution, such that $\widehat{X}$ and $\widehat{Y}$ have the same marginal probabilities as, respectively 
$X$ and $Y$. The advantage of a coupling is that we can tune the joint distribution to fit our needs. For our proof we construct
a coupling $(\widehat{\D}_n, \widehat{P}_n)$, such that
\begin{align*}
	\lim_{n \to \infty} \Prob{\widehat{\mathscr{D}}_n \ne \widehat{P}_n} = 0.
\end{align*}
Hence, since $\widehat{\D}_n$ and $\widehat{P}_n$ have the same distribution as, respectively $\mathscr{D}_n$ and $P_n$ we have
\begin{align*}
	\Prob{\mathscr{D}_n = k} 
	= \Prob{P_n = k} + \Prob{\widehat{\mathscr{D}}_n = k, \widehat{\mathscr{D}}_n \ne \widehat{P}_n},
\end{align*}
from which it follows that $|\Prob{\mathscr{D}_n = k} - \Prob{P_n = k}| \to 0$. Finally, we show that 
\begin{align*}
	\lim_{n \to \infty} \Prob{n \widehat{w}_n(X) > k} = \Prob{Y > k},
\end{align*}
where $Y$ is a Pareto random variable with shape $\nu \beta$ and scale $\gamma$, i.e. it has probability density $P_Y(y)$ given by
\eqref{eq:pareto-pdf}. This implies that the mixed Poisson
random variable $P_n$ with mixing parameter $n \widehat{w}_n(X)$ converges to a mixed Poisson $\mathscr{D}$ with mixing parameter $Y$,
which proves the result.

Before we give the proofs of the two theorems we first establish some technical results needed to construct the coupling, required
for Theorem \ref{thm:mixed_poisson_degrees_hscm}, the proof of which is given in Section \ref{sssec:proof_degree_distribution_hscm}.

\subsubsection{Technical results on Poisson couplings and concentrations}

We first establish a general result for couplings between mixed Poisson random variables, where the mixing parameters converge in
expectation.

\begin{lemma}\label{lem:coupling_mixed_poisson_converging_rvs}
Let $X_n$ and $Y_n$ be random variables such that,
\begin{align*}
	\lim_{n \to \infty} \Exp{|X_n - Y_n|} = 0.
\end{align*}
Then, if $P_n$ and $Q_n$ are mixed Poisson random variables
with, respectively, parameters $X_n$ and $Y_n$,
\begin{align*}
	\Prob{P_n \ne Q_n} = O\left(\Exp{|X_n - Y_n|}\right),	
\end{align*}
and in particular
\begin{align*}
	\lim_{n \to \infty} \Prob{P_n \ne Q_n} = 0.
\end{align*}
\end{lemma}

\begin{proof}
Let $a_n = \Exp{|X_n - Y_n|}$ and define the event $A_n = \{|X_n - Y_n| \le \sqrt{a_n}\}$. Then, since
\begin{align*}
	\lim_{n \to \infty} \Prob{|X_n - Y_n| > \sqrt{a_n}} \le \lim_{n \to \infty} \frac{\Exp{|X_n - Y_n|}}{\sqrt{a_n}} = 0,
\end{align*}
it is enough to show that
\begin{align*}
	\lim_{n \to \infty} \Prob{P_n \ne Q_n, A_n} = 0.
\end{align*}
Take $\hat{P}_n$ to be mixed Poisson with parameter $X_n + \sqrt{a_n}$. In addition, let $V_n$ and $Z_n$ be mixed Poisson
with parameter $\min\{X_n + \sqrt{a_n} - Y_n, 0\}$ and $\sqrt{a_n}$, respectively. Then, since on $A_n$ we have
$X_n + \sqrt{a_n} > Y_n$ we get, using Markov's inequality,
\begin{align*}
	\Prob{P_n \ne Q_n, A_n} &\le \Prob{P_n \ne Q_n, P_n = \hat{P}_n, A_n} + \Prob{\hat{P}_n \ne P_n, A_n}\\
	&= \Prob{\hat{P}_n \ne Q_n, A_n} + \Prob{Z_n \ne 0, A_n}
	= \Prob{V_n \ne 0, A_n} + \Prob{Z_n \ne 0, A_n} \\
	&\le \Prob{V_n \ge 1, A_n} + \Prob{Z_n \ge 1, A_n}
	\le \Exp{|X_n + \sqrt{a_n} - Y_n|} + \sqrt{a_n}\\
	&\le \Exp{|X_n - Y_n|} + 2 \sqrt{a_n} = O(\sqrt{a_n}). 
\end{align*}
Since $a_n \to 0$ by assumption, this finishes the proof.
\end{proof}

Next we show that $\widehat{W}(X, Y)$ converges in expectation to $\widehat{w}_n(X)$ when $X, Y$ have distribution
$\mu_n$. We also establish an upper bound on the rate of convergence, showing that it converges faster than $n^{-1}$.

\begin{lemma}\label{lem:concentration_W_hat_w}
Let $X, Y$ be independent random variables with density $\mu_n$. Then, as $n \to \infty$,
\begin{align*}
	\Exp{\left|\widehat{W}(X,Y) - \widehat{w}_n(X)\right|} = O\left(\log(n)n^{-\gamma} + n^{-\frac{1 + \gamma}{2}}\right)
\end{align*}
\end{lemma}

\begin{proof}
Recall that
\begin{align*}
	\widehat{w}_n(x) &= \begin{cases}
		\omega_n e^{-x} &\mbox{if } 0 \le x \le R_n\\
		0 &\mbox{otherwise}
	\end{cases}, 
	\text{ where }\\
    \omega_n &= \int_{0}^{R_n} e^{-x} \, d\mu_n(x) = \frac{1-e^{-(\gamma-1) R_n}}{\beta e^{R_n}}.
\end{align*}
Hence, it follows that
\begin{align*}
	\Exp{\left|\widehat{W}(X,Y) - \widehat{w}_n(X)\right|} 
	&= \int_{-\infty}^0 \int_{-\infty}^{R_n} \widehat{W}(x,y) \, d\mu_n(y) \, d\mu_n(x) \\
	&\hspace{10pt}+ \int_{0}^{R_n} \int_{-\infty}^{R_n} \left|\widehat{W}(x,y) - \widehat{w}_n(x)\right| \, d\mu_n(y) \, d\mu_n(x),
\end{align*}
where the first integral is $O(n^{-(\gamma + 1)/2})$ by Lemma \ref{lem:integration_W_W_hat}.

To deal with the second integral we first compute
\begin{align*}
	\int_{0}^{R_n} \int_{-\infty}^{-x} \, d\mu_n(y) \, d\mu_n(x) = \gamma R_n e^{-2 \gamma R_n} = O\left(\log(n) n^{-\gamma}\right).
\end{align*}
Therefore we have
\begin{align*}
	&\hspace{-30pt}\int_{0}^{R_n} \int_{-\infty}^{R_n} \left|\widehat{W}(x,y) - \widehat{w}_n(x)\right| \, d\mu_n(y) d\mu_n(x)\\
	&= \int_{0}^{R_n} \int_{-\infty}^{-x} \, d\mu_n(y) d\mu_n(x)
		+ \int_{0}^{R_n} \int_{-x}^{R_n} \left|e^{-(x + y)} - \omega_n e^{-x}\right| \, d\mu_n(y) d\mu_n(x)\\
	&\le O\left(\log(n) n^{-\gamma}\right)
		+ \int_0^{R_n} \int_0^{R_n} e^{-x}\left|e^{-y} - \omega_n\right| \, d\mu_n(y) d\mu_n(x) \\
	&= O\left(\log(n) n^{-\gamma}\right) + \omega_n \int_0^{R_n} \left|e^{-y} - \omega_n\right| \, d\mu_n(y) 
\end{align*}
We proceed to compute the last integral and show that it is $O(n^{-(\gamma + 1)/2})$. For this we note that since
\begin{align*}
	e^{-y} \le \omega_n \iff y \le -\log(\omega_n),
\end{align*}
we have
\begin{align*}
	\int_0^{R_n} \left|e^{-y} - \omega_n\right| \, d\mu_n(y)
	&= \int_0^{-\log(\omega_n)} (\omega_n - e^{-y}) \, d\mu_n(y) + \int_{-\log(\omega_n)}^{R_n} (e^{-y} - \omega_n) \, d\mu_n(y)
\end{align*}
For the first integral we compute
\begin{align*}
	\int_0^{-\log(\omega_n)} (\omega_n - e^{-y}) \, d\mu_n(y)
	&= \gamma e^{-\gamma R_n} \left(\frac{\omega_n^{1-\gamma}}{\gamma} - \frac{\omega_n^{1 - \gamma}}{\gamma - 1}
		- \frac{\omega_n}{\gamma} + \frac{1}{\gamma -1}\right) \\
	&= \frac{e^{-\gamma R_n}}{\beta} - \frac{\omega_n e^{-\gamma R_n}}{\gamma - 1} - \omega_n e^{-\gamma R_n} 
	= \frac{e^{-\gamma R_n}}{\beta} - \frac{\omega_n e^{-\gamma R_n}}{\beta}.
\end{align*}
Similar calculations yield
\begin{align*}
	\int_{-\log(\omega_n)}^{R_n} (e^{-y} - \omega_n) \, d\mu_n(y)
	= \frac{e^{-R_n}}{\beta} - \omega_n - \frac{\omega_n^{1-\gamma} e^{-\gamma R_n}}{\gamma - 1},
\end{align*}
and hence, 
\begin{align*}
	\int_0^{R_n} \left|e^{-y} - \omega_n\right| \, d\mu_n(y)
	&= \frac{1}{\beta}\left(e^{-R_n} + e^{-\gamma R_n}\right) - \omega_n -  \frac{\omega_n e^{-\gamma R_n}}{\beta}
		-  \frac{\omega_n^{1-\gamma} e^{-\gamma R_n}}{\gamma - 1}\\
	&\le \frac{1}{\beta}\left(e^{-R_n} + e^{-\gamma R_n}\right) - \omega_n. 
\end{align*}
We now use this last upper bound, together with
\begin{align*}
	\omega_n = \frac{1-e^{-(\gamma-1) R_n}}{\beta e^{R_n}} = \frac{1}{\beta}\left(e^{-R_n} - e^{-\gamma R_n}\right),
\end{align*}
to obtain
\begin{align*}
	 \omega_n \int_0^{R_n} \left|e^{-y} - \omega_n\right| \, d\mu_n(y)
	 &\le \frac{\omega_n}{\beta}\left(e^{-R_n} + e^{-\gamma R_n}\right) - \omega_n^2 \\
	 &= \frac{1}{\beta^2}\left(e^{-2R_n} + e^{-\gamma 2R_n}\right)
	 	- \frac{1}{\beta^2}\left(e^{-R_n} - e^{-\gamma R_n}\right)^2\\
	 &= \frac{2}{\beta^2} e^{-(\gamma + 1)R_n} = O\left(n^{-\frac{1 + \gamma}{2}}\right),
\end{align*}
from which the result follows.
\end{proof}

\subsubsection{Proof of Theorem \ref{thm:mixed_poisson_degrees_hscm}}\label{sssec:proof_degree_distribution_hscm}

We start with constructing the coupling $(\widehat{\mathscr{D}}_n, \widehat{P}_n)$, between 
$\mathscr{D}_n$ and $P_n$ such that
\begin{equation}\label{eq:coupling_D_P}
	\lim_{n \to \infty} \Prob{\widehat{\mathscr{D}}_n \ne \widehat{P}_n} = 0.
\end{equation}
First, let $I_{j}$ be the indicator that the edge $(1,j)$ is present in $G_n$. Let ${\bf X}_n = X_1, \dots, X_n$ denote the
coordinates of the nodes. Then, conditioned on these, we have that $I_{j}$ are independent Bernoulli random variables with probability
$W(X_1, X_j)$, while $D_1 = \sum_{j = 2}^n I_{j}$. Now, let $Q_n$ be a
mixed Poisson with parameter $\sum_{j = 2}^n W(X_1, X_j)$. Then, see for instance \cite[Theorem 2.10]{hofstad2016random},
there exists a coupling $(\widehat{D}_1, \widehat{Q}_n)$ of $D_1$ and $Q_n$, such that
\begin{align*}
	\CProb{\widehat{D}_1 \ne \widehat{Q}_n}{{\bf X}_n} \le \sum_{j = 2}^n W(X_1, X_j)^2.
\end{align*}
Therefore, we get that
\begin{align*}
	\Prob{\widehat{D}_1 \ne \widehat{Q}_n}
	&\le (n - 1) \iint_{-\infty}^{R_n} W(x, y)^2 \, d\mu_n(y) \, d\mu_n(x) \\
	&\le n \left(\int_{-\infty}^{R_n} e^{-2x} \, d\mu_n(x)\right)^2
	= \begin{cases}
			O\left(n^{-(\gamma - 1)}\right) &\mbox{if } 1 < \gamma < 2\\
			O\left(\log(n) n^{-1}\right) &\mbox{if } \gamma = 2\\
			O\left(n^{-1}\right) &\mbox{if } \gamma > 2
		\end{cases}. 
\end{align*}

Next, since $X_1$ and $X_j$ are independent for all $2 \le j \le n$ we use Proposition \ref{prop:concentration_W_W_hat} together with 
Lemma \ref{lem:concentration_W_hat_w} to obtain
\begin{align*}
	&\hspace{-10pt}\Exp{\left|\sum_{j = 2}^n W(X_1,X_j) - n \widehat{w}_n(X_1)\right|}\\
	&\le (n-1)\Exp{\left|W(X_1,X_2) - \widehat{w}_n(X_1)\right|} + \Exp{\widehat{w}_n(X_1)}\\
	&\le (n - 1)\Exp{\left|W(X_1, X_2) - \widehat{W}(X_1, X_2)\right|}
		+ (n - 1)\Exp{\left|\widehat{W}(X_1, X_2) - \widehat{w}_n(X_1)\right|} + \Exp{\widehat{w}_n(X_1)} \\
	&= O\left(n^{-\frac{\gamma - 1}{2}}\right).
\end{align*}
from which it follows that
\begin{equation}\label{eq:convergence_W_n_w_hat}
	\lim_{n \to \infty} \Exp{\left|\sum_{j = 2}^n W(X_1,X_j) - n \widehat{w}_n(X_1)\right|} = 0.
\end{equation}
Now let $\widehat{P}_n$ be a mixed Poisson random variable with mixing parameter $n \widehat{w}_n(X_1)$. Then by 
\eqref{eq:convergence_W_n_w_hat} and Lemma 
\ref{lem:coupling_mixed_poisson_converging_rvs}
\begin{align*}
	\lim_{n \to \infty} \Prob{\widehat{P}_n \ne \widehat{Q}_n} = 0
\end{align*}
and \eqref{eq:coupling_D_P} follows from
\begin{align*}
	\Prob{\widehat{\mathscr{D}}_n \ne \widehat{P}_n} \le \Prob{\widehat{\mathscr{D}}_n \ne \widehat{Q}_n}
	+ \Prob{\widehat{P}_n \ne \widehat{Q}_n}.
\end{align*}
As a result we have that
\begin{align*}
	\lim_{n \to \infty} \left|\Prob{\mathscr{D}_n = k} - \Prob{P_n = k}\right|
	&= \lim_{n \to \infty} \left|\Prob{\widehat{\mathscr{D}}_n = k} - \Prob{\widehat{P}_n = k}\right| \\
	&= \lim_{n \to \infty} \Prob{\widehat{\mathscr{D}}_n = k, \widehat{\mathscr{D}}_n \ne \widehat{P}_n}
	\le \lim_{n \to \infty} \Prob{\widehat{\mathscr{D}}_n \ne \widehat{P}_n} = 0, \numberthis \label{eq:convergence_D_P_distribution}
\end{align*}

We will now prove that if $X$ has density $\mu_n$, then for any $t > 0$
\begin{equation}\label{eq:convergence_degree_mixing_paramters}
	\lim_{n \to \infty} \left|\Prob{n\widehat{w}_n(X) > t} - \Prob{Y > t}\right| = 0.
\end{equation}
A sequence $\{P_n\}_{n \ge 1}$ of mixed Poisson random variables, with mixing parameters $Z_n$ converges to a mixed Poisson random
variable $P$ with mixing parameter $Z$, if $Z_n$ converge to $Z$ in distribution, see for instance \cite{grandell1997mixed}.
Therefore, since $\mathscr{D}$ is mixed Poisson with mixing parameter $Y$, \eqref{eq:convergence_degree_mixing_paramters} implies
\begin{align*}
	\lim_{n \to \infty} \Prob{P_n = k} = \Prob{\mathscr{D} = k} 
\end{align*}
which, combined with \eqref{eq:convergence_D_P_distribution}, yields
\begin{align*}
	\lim_{n \to \infty} \Prob{\mathscr{D}_n = k} = \Prob{\mathscr{D} = k}.
\end{align*}

To establish \eqref{eq:convergence_degree_mixing_paramters}, first define
\begin{align*}
	\epsilon_n = e^{-(\gamma - 1)R_n} - e^{-2\gamma R_n}
\end{align*}
so that 
\begin{align*}
	n \widehat{w}_n(x) = \frac{n}{\beta e^{R_n}} \left(1 - e^{-(\gamma - 1)R_n}\right) = \sqrt{\nu n} \, e^{-x} \left(1 - \epsilon_n\right).
\end{align*}
Next, for all $0 \le x \le R_n$ we have that
\begin{align*}
	\nu \beta \left(1 - \epsilon_n\right) \le n \widehat{w}_n(x) \le \sqrt{\nu n} \, \left(1 - \epsilon_n\right)
\end{align*}
and hence
\begin{align*}
	\Prob{n \hat{w}_n(X) > t} = \begin{cases}
		\Prob{X < \log\left(\frac{n \omega_n}{t}\right)} &\mbox{if } \nu \beta \left(1 - \epsilon_n\right) \le t \le \sqrt{\nu n} \, \left(1 - \epsilon_n\right) \\
		1 &\mbox{if } t < \nu \beta \left(1 - \epsilon_n\right) \\
		0 &\mbox{else.}
	\end{cases}
\end{align*}
Moreover, for any $\nu \beta \left(1 - \epsilon_n\right) \le t \le \sqrt{\nu n} \, \left(1 - \epsilon_n\right)$
\begin{align*}
	\Prob{n\widehat{w}_n(X) > t} &= \Prob{X < \log\left(\frac{n \omega_n}{t}\right)}
	= \int_{-\infty}^{\log\left(\frac{n \omega_n}{t}\right)} \, d\mu_n(x) \\
	&= e^{-\gamma R_n}\left(\frac{n \omega_n}{t}\right)^{\gamma} 
	= \left(n \omega_n e^{-R_n}\right)^\gamma t^{-\gamma}
	= \left(\frac{\nu \beta}{t}\right)^{\gamma} (1 - \epsilon_n)^\gamma. \numberthis \label{eq:tail_probability_n_hat_w_X}
\end{align*}

Now, fix $t < \nu \beta$. Then, for large enough $n$ it holds that $t < \nu \beta (1 - \epsilon_n)$ in which case \eqref{eq:convergence_degree_mixing_paramters}
holds trivially, since both probabilities are $1$. Hence we can assume, without loss of generality that $t \ge \nu \beta$. Then for $n$ large enough such $t \le \sqrt{\nu n} \, \left(1 - \epsilon_n\right)$, it follows from \eqref{eq:tail_probability_n_hat_w_X} and \eqref{eq:pareto_tail_cdf} that
\begin{align*}
	\left|\Prob{n\widehat{w}_n(X) > t} - \Prob{Y > t}\right|
	&= \left|\left(\frac{\nu \beta}{t}\right)^{\gamma} (1 - \epsilon_n)^\gamma - \left(\frac{\nu \beta}{t}\right)^\gamma\right|\\
	&= \left(\frac{\nu \beta}{t}\right)^{\gamma} \left(1 - (1 - \epsilon_n)^\gamma\right) \le 1 - (1 - \epsilon_n)^\gamma,
\end{align*}
which converges to zero as $n \to \infty$ and hence proves \eqref{eq:convergence_degree_mixing_paramters}.

\subsubsection{Proof of Theorem \ref{thm:average_degree_hscm}}\label{sssec:proof_average_degree_hscm}
 
First using Lemma \ref{lem:integration_W_W_hat}, we compute that
\begin{align*}
	(n-1)\Exp{\widehat{W}(X,Y)} 
	&=(n - 1) \iint_{-\infty}^{R_n} \widehat{W}(x,y) \, d\mu_n(y) \, d\mu_n(x)\\
	&= (n - 1)\iint_0^{R_n} e^{-(x + y)} \, d\mu_n(y) \, d\mu_n(x) + O\left(n^{-\frac{\gamma - 1}{2}}\right)\\
	&= \frac{(n-1)}{\beta^2} e^{-2 R_n}\left(1 - e^{-\gamma R_n}\right)^2 + O\left(n^{-\frac{\gamma - 1}{2}}\right)\\
	&= \nu \left(1 - \frac{1}{\beta^2 n}\right)\left(1 - e^{-\gamma R_n}\right)^2 + O\left(n^{-\frac{\gamma - 1}{2}}\right)
	= \nu + O\left(n^{-\frac{\gamma - 1}{2}}\right), 
\end{align*}
where the last line follows since $-1 < -\gamma/2 < -(\gamma-1)/2$.

Next recall \eqref{eq:average_degree_hscm}
\begin{align*}
	\Exp{\mathscr{D}_n} = \sum_{j = 2}^n \Exp{W(X_1, X_j)} = (n - 1)\Exp{W(X,Y)}.
\end{align*}
Therefore, using Lemma \ref{lem:concentration_W_hat_w} we have
\begin{align*}
	\left|\Exp{\mathscr{D}_n} - \nu\right| &= \left|(n - 1)\Exp{W(X,Y)} - \nu\right|\\
	&\le (n-1)\Exp{\left|W(X,Y) - \widehat{W}(X,Y)\right|} + (n-1)\left|\Exp{\widehat{W}(X,Y)} - \nu\right|\\
	&\le O\left(\log(n)n^{-(\gamma - 1)}\right) + O\left(n^{-\frac{\gamma - 1}{2}}\right), 
\end{align*}
which yields the result.

\subsection{Proofs for graphon entropy}\label{ssec:proof_graphon_entropy}

Here we derive the properties of the graphon entropy $\sigma$ of our graphon $W$ given by \eqref{eq:def_hscm_graphon}. We first give the
proof of Proposition \ref{prop:max_entropy_graphon}, and then that of Theorem \ref{thm:convergence_graphon_entropy}.

\subsubsection{Proof of Proposition \ref{prop:max_entropy_graphon}}\label{sssec:proof_max_entropy_graphon}

Recall that, given a measure $\mu$ on some interval $A \subseteq \R$ with $\mu(A) < \infty$ and a $\mu$-integrable function $w(x)$, we consider the problem of maximizing 
$\sigma[W, \mu]$ under the constraint \eqref{eq:def_graphon_entropy_constraint},
\begin{align*}
	w(x) = \int_A W(x,y) \, d\mu(y).
\end{align*}
In particular we need to show that if a solution exists, it is given by \eqref{eq:def_max_entropy_graphon_solution}. Therefore, suppose
there exists at least one graphon $W$ which satisfies the constraint. For this we use the technique of Lagrange multipliers from variation 
calculus \cite{aubin2000applied,gelfand2000calculus}.

To set up the framework, let $\mathcal{W}$ denote the space of symmetric functions $W: A \times A \to [0,1]$ which satisfy 
\begin{align*}
	\iint_{A^2} W(x,y) \, d\mu(y) \, d\mu(x) < \infty.
\end{align*}
Observe that $\mathcal{W}$ is a convex subset of the Banach space $\mathcal{W}_{\mathbb{R}}$, of all symmetric, $\mu$-integrable, functions
$W: A \times A \to \mathbb{R}$ and that the function $w$ is an element of $L^1(-\infty,R)$ with respect to the measure $\mu$,
which is also Banach. Denote the latter space by $L^1_{A, \mu}$. We slightly abuse notation and write $\sigma$ for the
functional 
\begin{align*}
	W \mapsto \iint_{A^2} H(W(x,y)) \, d\mu(y) \, d\mu(x) = \sigma[W,\mu],
\end{align*} 
and define the functional 
$\mathcal{F} : \mathcal{W}_{\mathbb{R}} \to L^1_{A, \mu}$ by
\begin{align*}
	\mathcal{F}(W)(x) = w(x) - \int_A W(x,y) \, d\mu(y).
\end{align*}
Then we need to solve the following Euler-Lagrange equation for some Lagrange multiplier functional $\lambda: L^1_{A, \mu} \to \mathbb{R}$,
\begin{align*}
	\frac{\partial}{\partial W} \left(\sigma(W) - \Lambda \circ \mathcal{F}(W)\right) = 0,
\end{align*}
with respect to the Fr\'{e}chet derivative. By Riesz Representation Theorem, we
have that for any functional $\Lambda: L^1_{A, \mu} \to \mathbb{R}$, there
exists a $\lambda \in L^\infty$, uniquely defined on $A$, such that
\begin{align*}
	\Lambda(f) = \int_A \lambda(x) f(x) \,d\mu(x).
\end{align*}
Hence our Euler-Lagrange equation becomes
\begin{equation}\label{eq:initial_graphon_lagrangian}
	\frac{\partial}{\partial W} \left(\sigma(W) - \int_A\lambda(x)\left(w(x) - \int_A W(x,y) \, d\mu(y)\right)\,d\mu(x)\right) = 0.
\end{equation}
In addition, since $W$ is symmetric we have
\begin{align*}
	\iint_{A^2} \lambda(x) W(x,y) \, d\mu(y) d\mu(x)
	= \frac{1}{2} \iint_{A^2} (\lambda(x) + \lambda(y)) W(x,y) \, d\mu(y) \, d\mu(x)
\end{align*}
and hence, by absorbing the factor $1/2$ into $\lambda$ we can rewrite our Euler-Lagrange equation as
\begin{equation}\label{eq:graphon_lagrangian}
	\frac{\partial}{\partial W} \left(\sigma(W) - \int_A \lambda(x)w(x) \, d\mu(x)
	+ \iint_{A^2} (\lambda(x) + \lambda(y))W(x,y) \, d\mu(y) \, d\mu(x)\right) = 0.
\end{equation}
For the two derivatives we have
\begin{align*}
	&\frac{\partial \sigma(W)}{\partial W} = \log(1 - W(x,y)) - \log(W(x,y))
		= -\log\left(\frac{W(x,y)}{1 - W(x,y)}\right),\\
	&\frac{\partial}{\partial W} \left(\int_A \lambda(x)w(x) \, d\mu(x)
		+ \iint_{A^2} (\lambda(x) + \lambda(y))W(x,y) \, d\mu(y) \, d\mu(x)\right)
		= \lambda(x) + \lambda(y). \numberthis \label{eq:derivative_graphon_contraint}
\end{align*}
Hence, we need to solve the equation
\begin{align*}
	-\log\left(\frac{W(x,y)}{1 - W(x,y)}\right) = \lambda(x) + \lambda(y),
\end{align*}
which gives \eqref{eq:def_max_entropy_graphon_solution}.

There is, however, a small technicality related to the computation of the derivative \eqref{eq:derivative_graphon_contraint}. This is caused by the 
fact that $H$ is only defined on $[0,1]$. In particular, for $W_1, W_2 \in \mathcal{W}$, it could be that $W_1 + W_2 > 1$
on some subset $C \subseteq A \times A$ and hence $H(W_1 + W_2)$ is not well defined. To compute
the Fr\'{e}chet derivative we need $H(W_1 + \varepsilon W_2)$ to be well defined for any $W_1, W_2 \in
\mathcal{W}$ and some, small enough, $\varepsilon > 0$. To this end we fix a $0 < \delta < 1$ and define
\begin{align*}
	H_\delta(x) = H\left((1 - \delta)x + \frac{\delta}{2}\right),
\end{align*}
which is just $H$, stretched out to the interval $(-\delta/2(1 - \delta), 1 + \delta/2(1 - \delta))$.
Similarly we define $\sigma_\delta$, using $H_\delta$ and consider the corresponding graphon
entropy maximization problem. Then we can compute $\partial \sigma_\delta /\partial W$, by taking
a $W^\prime \in \mathcal{W}$ and $\varepsilon > 0$ such that $W + \varepsilon W^\prime
<  1 + \delta/2(1 - \delta))$. Then $H_\delta(W + \varepsilon W^\prime)$ is well-defined and
using the chain rule we obtain
\begin{align*}
	\frac{\partial H_\delta(W + \varepsilon W^\prime) }{\partial \varepsilon}
	= -(1 - \delta) W^\prime \log\left(\frac{(1 - \delta)(W + \varepsilon W^\prime) + \delta/2}
	{1 - (1 - \delta)(W + \varepsilon W^\prime) - \delta/2}\right), 
\end{align*}
from which it follows that
\begin{align*}
	\frac{\partial \sigma_\delta(W)}{\partial W} = -(1 - \delta)\log\left(\frac{(1 - \delta)W + \delta/2}
	{1 - (1 - \delta)W - \delta/2}\right).
\end{align*}
Therefore we have the following equation
\begin{align*}
	-(1 - \delta) \log\left(\frac{(1 - \delta)(W + \delta/2}
	{1 - (1 - \delta)W - \delta/2}\right)
	= \lambda(x) + \lambda(y),
\end{align*}
which leads to a solution $W_\delta$ of the form
\begin{align*}
	W_\delta(x,y) = \frac{1 - \frac{\delta}{2}
		\left(1 + e^{\frac{\lambda(x) + \lambda(y)}{(1 - \delta)}}\right)}
		{(1-\delta)\left(1
		+ e^{\frac{\lambda(x) + \lambda(y)}{(1 - \delta)}}\right)}.
\end{align*}
Since $\delta < 1$ it follows, using elementary algebra, that $W_\delta(x,y) \in [0,1]$
for all $x,y \in [0,1]$. From this we conclude that $W_\delta \in \mathcal{W}$. Moreover
$W_\delta$ converges to \eqref{eq:def_max_entropy_graphon_solution} as $\delta \to 0$. Since
$\sigma_\delta \to \sigma$ as $\delta \to 0$, we obtain the graphon that maximizes the
entropy $\sigma$, where the function $\lambda$ is determined by the constraint
\eqref{eq:def_graphon_entropy_constraint}.

For uniqueness, suppose there exist two solutions, $\lambda_1$ and $\lambda_2$ in $L_{A,\mu}^1$, to \eqref{eq:def_graphon_entropy_constraint}
and for which the graphon entropy is maximized. Let $f(x) = \lambda_1(x) - \lambda_2(x) \in L_{A, \mu}^1$, so that $\lambda_1 = \lambda_2 + f$. Since $\lambda_1$
satisfies \eqref{eq:initial_graphon_lagrangian} it follows that, due to
linearity of the derivative,
\begin{align*}
	\frac{\partial}{\partial W} \int_{A} f(x)W(x,y) \, d\mu(y) \, d\mu(x) = 0.
\end{align*}
Now since
\begin{align*}
	\frac{\partial}{\partial W} \int_{A} f(x)W(x,y) \, d\mu(y) \, d\mu(x) = f(x),
\end{align*}
it follows that $f = 0$, $\mu$ almost everywhere on $A$ and hence $\lambda_1 = \lambda_2$, $\mu$ almost everywhere on $A$.  

\subsubsection{Proof of Theorem \ref{thm:convergence_graphon_entropy}}\label{sssec:proof_convergence_graphon_entropy}

First note that Proposition \ref{prop:concentration_Gibbs_W_W_hat} implies that the difference in expectation between 
$H(W)$ and $H(\widehat{W})$ converges to zero faster than $\log(n)/n$, so that for the purpose of Theorem 
\ref{thm:convergence_graphon_entropy} we can approximate $W$ with $\widehat{W}$. Hence we are left to show that 
the rescaled entropy $n \sigma[\widehat{W}, \mu_n]/\log(n)$ converges to $\nu$. By Lemma \ref{lem:integration_H_W_W_hat} 
the integration over all regimes except $[0, R_n]^2$ goes to zero faster than $\log(n)/n$ and therefore we
only need to consider integration over $[0, R_n]^2$. The main idea for the rest of the proof is that in this range
\begin{align*}
	H\left(\widehat{W}(x,y)\right) \approx (x+y)\widehat{W}(x,y).
\end{align*}

Let us first compute the integral over $[0, R_n]^2$ of the right hand side in the above equation.
\begin{align*}
	&\hspace{-20pt} \iint_0^{R_n} (x + y) \widehat{W}(x,y) \, d\mu_n(y) \, d\mu_n(x)\\
	&= \gamma^2 e^{-2\gamma R_n}\iint_0^{R_n} (x + y) e^{(\gamma - 1)(x + y)} \, dy \, dx \\
	&= \frac{2 e^{-2(1 + \gamma) R_n}}{\beta^2 (\gamma - 1)}\left((\gamma - 1)R_n \left(e^{2 \gamma R_n} + e^{(\gamma + 1)R_n}\right)
		- e^{2 R_n} - e^{2 \gamma R_n}\right) \\
	&= \frac{2 R_n }{\beta^2} \left(e^{-2R_n} + e^{-(\gamma + 1)R_n}\right)
		- \frac{2}{ \beta^2(\gamma - 1)} \left(e^{-2\gamma R_n} + e^{-2R_n}\right)\\
	&= \frac{2 R_n e^{-2R_n}}{\beta^2} + O(n^{-1})
	= \nu n^{-1} \log(n) + O(n^{-1}) 
\end{align*}
which implies that
\begin{equation}\label{eq:convergence_approx_W_hat_entropy}
	\lim_{n \to \infty} \frac{n}{\log(n)} \iint_0^{R_n} (x + y) \widehat{W}(x,y) \, d\mu_n(y) \, d\mu_n(x) = \nu.
\end{equation}

Next we show that
\begin{equation}\label{eq:bound_diff_H_hat_W_approx}
	\frac{n}{\log(n)}\left|\iint_0^{R_n} \left(H\left(\widehat{W}(x,y)\right) -(x + y)\widehat{W}(x,y)\right) \, d\mu_n(y) \, d\mu_n(x)\right|
	= O\left(\log(n)^{-1}\right),
\end{equation}
which, together with \eqref{eq:convergence_approx_W_hat_entropy}, gives
\begin{equation}\label{eq:convergence_W_hat_entropy_0_Rn}
	\lim_{n \to \infty} \frac{n}{\log(n)}\iint_0^{R_n} H\left(\widehat{W}(x,y)\right) \, d\mu_n(y) \, d\mu_n(x) = \nu.
\end{equation}
We compute that
\begin{align*}
	H\left(\widehat{W}(x,y)\right) &= e^{-(x+y)}(x + y)
		- (1 - e^{-(x + y)})\log\left(\frac{e^{x + y} - 1}{e^{x + y}}\right)\\
	&= (x + y)\widehat{W}(x,y) - (1 - e^{-(x + y)})\log\left(1-e^{-(x + y)}\right). \numberthis \label{eq:diff_H_W_hat_W}
\end{align*}
Note that $- (1 - e^{-z})\log\left(1-e^{-z}\right) \le e^{-z}$ for all $z \ge 0$. Hence, it follows from \eqref{eq:diff_H_W_hat_W} that, on $[0, R_n]^2$,
\begin{align*}
	\left|H\left(\widehat{W}(x,y)\right) -(x + y)\widehat{W}(x,y)\right|
	&= -\left(1 - e^{-(x + y)}\right)\log\left(1 - e^{-(x + y)}\right) \le e^{-(x+y)} = \widehat{W}(x,y),
\end{align*}
so that by Theorem \ref{thm:average_degree_hscm},
\begin{align*}
	&\hspace{-30pt}\frac{n}{\log(n)}\left|\iint_0^{R_n} \left(H\left(\widehat{W}(x,y)\right) -(x + y)\widehat{W}(x,y)\right) \, d\mu_n(y) \, d\mu_n(x)\right|\\
	&\le \frac{n}{\log(n)}\iint_0^{R_n} \widehat{W}(x,y) \, d\mu_n(y) \, d\mu_n(x) = O\left(\log(n)^{-1}\right).
\end{align*}

To conclude, we have
\begin{align*}
	\sigma[\widehat{W},\mu_n] &= \iint_{-\infty}^{R_n} H\left(\widehat{W}(x,y)\right) \, d\mu_n(y) d\mu_n(x) \\
	&= \iint_{0}^{R_n} H\left(\widehat{W}(x,y)\right) \, d\mu_n(y) d\mu_n(x) + O\left(n^{-\gamma} + \log(n)n^{-\frac{\gamma + 1}{2}}\right)
\end{align*}
and hence, using \eqref{eq:convergence_W_hat_entropy_0_Rn},
\begin{align*}
	\lim_{n \to \infty} \frac{n \sigma[\widehat{W}, \mu_n]}{\log(n)} = \nu.
\end{align*}

\subsection{Proof of Theorem \ref{thm:convergence_graph_entropy}}\label{ssec:proof_convergence_graph_entropy}

In this section we first formalize the strategy behind the proof of Theorem \ref{thm:convergence_graph_entropy},
briefly discussed in Section~\ref{ssec:convergence_graph_entropy}. This strategy relies on partitioning the interval $A_n$
into non-overlapping subintervals. We then construct a specific partition satisfying certain requirements, and
finish the proof of the theorem.

\subsubsection{Averaging $W$ by a partition of $A_n$}

We follow the strategy from \cite{janson2013graphons}. First recall that for a graph
$G_n$ generated by the HSCM,
\begin{align*}
	\mathcal{S}[G_n] \ge \Exp{\mathcal{S}[G_n| {\bf x}_n]} = \binom{n}{2} \Exp{H(W(x_1,x_2))} = \binom{n}{2} \sigma[G_n],
\end{align*}
and hence 
\begin{align*}
	\mathcal{S}^\ast[G_n] \ge \sigma[G_n],
\end{align*}
where $\mathcal{S}^\ast[G_n] = \mathcal{S}[G_n]/\binom{n}{2}$ denotes the normalized Gibbs entropy.

Therefore, the key ingredient is to find a matching upper bound. For this we partition the range
$A_n = (-\infty, R_n]$ of our probability measure $\mu_n(x)$ into intervals and approximate $W(x,y)$ by
its average over the box in which $x$ and $y$ lie.

To be more precise, let $m_n$ be any increasing sequence of positive integers, $\{\rho_t\}_{0 \le t \le m_n}$ be such that
\begin{align*}
	\rho_0 = -\infty, \quad \text{and} \quad \rho_{m_n} = R_n,
\end{align*}
and consider the partition of $(-\infty, R_n]$ given by
\begin{align*}
	I_t = (\rho_{t-1},\rho_t] \quad \text{for } 1 \le t \le m_n.
\end{align*}

Now define $J_n(x) = t \iff x \in I_t$, and let $M_i$ be the random variable $J_n(X_i)$, where $X_i$ has density function
$\mu_n$ for any vertex $i$. The value of $M_i$ equal to $t$ indicates that vertex $i$ happens to lie within interval $I_t$.
Denoting the vector of these random variable by $\textbf{M}_n=M_1,\ldots,M_n$, and their entropy by $\mathcal{S}[\textbf{M}_n]$,
we have that
\begin{align*}
	\mathcal{S}[G_n] &\le \mathcal{S}[\textbf{M}_n] + \mathcal{S}[G_n|\textbf{M}_n]\\
	&\le n \mathcal{S}[M] + \binom{n}{2} \iint_{-\infty}^{R_n} H(\widetilde{W}(x,y)) \, d\mu_n(y) \ d\mu_n(x),
\end{align*}
where $\widetilde{W}(x,y)$ is the average of $W$ over the square $I_{J_n(x)} \times I_{J_n(y)}$. That is,
\begin{equation}\label{eq:def_W_tilde}
	\widetilde{W}_n(x,y) = \frac{1}{\mu_{x, y}}
	\int_{I_{J_n(x)}}\int_{I_{J_n(y)}} W(u,v) \, d\mu_n(v) \, d\mu_n(u),
\end{equation}
with
\begin{align*}
	\mu_{x, y} = \int_{I_{J_n(x)}} \int_{I_{J_n(y)}} \, d\mu_n(v) \, d\mu_n(u),
\end{align*}
the measure of the box to which $(x, y)$ belongs.

The first step in our proof is to investigate how well $\widetilde{W}_n$ approximates $W$. More
specifically, we want to understand how the difference $|\sigma[W, \mu_n] - \sigma[\widetilde{W}_n, \mu_n]|$
scales, depending on the specific partition. Note that for any partition $\rho_t$ of $A_n$ into $m_n$ intervals we have that
\begin{align*}
	\mathcal{S}[M] = - \sum_{t = 1}^{m_n} \Prob{M = t}\log\left(\Prob{M = t}\right) \le \log(m_n),
\end{align*}
where the upper bound is achieved on the partition which is uniform according to measure $\mu_n$.
Since
\begin{align*}
	\frac{n}{\log(n)} \frac{n \mathcal{S}[M]}{\binom{n}{2}} = \frac{2\mathcal{S}[M]}{\log(n)(1 - 1/n)},
\end{align*}
it is enough to find a partition $\rho_t$, with $\log(m_n) = o(\log(n))$, such that
\begin{equation}\label{eq:def_convergence_scaled_Gibbs_graphon_entropy_full}
	\lim_{n \to \infty} \frac{n}{\log(n)}|\sigma[W, \mu_n] - \sigma[\widetilde{W}_n, \mu_n]| = 0,
\end{equation}
This then proves Theorem \ref{thm:convergence_graph_entropy}, since
\begin{align*}
	\lim_{n \to \infty} \frac{n}{\log(n)}\left|\mathcal{S}^\ast[G_n] - \sigma[G_n]\right|
	&= \lim_{n \to \infty} \frac{n}{\log(n)} \left(\mathcal{S}^\ast[G_n] - \sigma[G_n]\right) \\
	&\le \lim_{n \to \infty} \left(\frac{n}{\log(n)} \frac{n \mathcal{S}[M]}{\binom{n}{2}} 
		+ \frac{n}{\log(n)}|\sigma[W, \mu_n] - \sigma[\widetilde{W}_n, \mu_n]|\right) \\
	&\le \lim_{n \to \infty} \left(\frac{2\log(m_n)}{\log(n)(1 - 1/n)}
		+ \frac{n}{\log(n)}|\sigma[W, \mu_n] - \sigma[\widetilde{W}_n, \mu_n]|\right) = 0. 
\end{align*}

\subsubsection{Constructing the partition}

We will take $I_1 = (-\infty, -R_n]$ and partition the remaining interval $[-R_n, R_n]$ into $\log(n)^2$ equal parts. To this end, let $n$ be sufficiently large so that $n \ge \nu \beta^2$, take 
$m_n = \lceil \log(n)^2 \rceil + 1$, and define the partition $\rho_t$ by
\begin{align*}
	\rho_0 = -\infty, \quad \rho_1 = -R_n, \quad \text{and} \quad \rho_{t} = \rho_{t-1} + \frac{2R_n}{m_n - 1} \quad \text{for all } t= 2, \dots, m_n.
\end{align*}

Note that $\log(m_n) = O(\log\log n) = o(\log(n)$ so that all that is left, is to prove \eqref{eq:def_convergence_scaled_Gibbs_graphon_entropy_full}. In addition,
\[
	\frac{n}{\log(n)} \int_{-\infty}^{-R_n} \int_{-\infty}^{R_n} H(\widetilde{W}(x,y)) \, d\mu_n(y) \ d\mu_n(x) = O\left(\log(n)^{-1} n^{1-\gamma}\right),
\]
and the same holds if we replace $\widetilde{W}$ with $W$. Hence it follows that in order to establish \eqref{eq:def_convergence_scaled_Gibbs_graphon_entropy_full} we only 
need to consider the integral over the square $[-R_n, R_n] \times [-R_n, R_n]$. That is, we need to show
\begin{equation}\label{eq:def_convergence_scaled_Gibbs_graphon_entropy}
	\lim_{n \to \infty} \frac{n}{\log(n)}\left|\iint_{-R_n}^{R_n} \left(H\left(W(x,y)\right) - H\left(\widetilde{W}(x,y)\right)\right) \, d\mu_n(y) \, d\mu_n(x)\right| = 0,
\end{equation}

For this we compare $\sigma[W,\mu_n] $ and 
$\sigma[\widetilde{W},\mu_n]$, based on the mean value theorem which states that
\begin{align*}
	\text{for any } a \leq b, \quad \exists \, c\in[a,b], \quad \text{such that } H(b)-H(a) = H'(c) \, (a-b). 
\end{align*}
Here $|H'(c)|  = |\log(c) - \log(1-c)|$, and due to the symmetry of $H$ we get, for any $0 < a \le c \le b < 1$,
\begin{equation}
\label{eq:H'minbound} |H'(c)| \leq \frac{\min\{H(a),H(b)\}}{\min\{a,b,1-a,1-b\}}.
\end{equation}

Note that
\begin{equation}\label{eq:partition_bounds_W_tilde_W}
	0 < \min_{u \in I_{M_n(x)}, \, v \in I_{M_n(y)}} W(u,v) \le \widetilde{W}_n(x,y)
	\le \max_{u \in I_{M_n(x)}, \, v \in I_{M_n(y)}} W(u,v) < 1,
\end{equation}
for all $x, y \in [-R_n, R_n]$. In addition, for all $x, y \in [-R_n, R_n]$ and $(u,v) \in I_{M_n(x)} \times I_{M_n(y)}$ we have $|x+y-u-v| \leq 2R_n/m_n \leq 2/\log(n)$
and thus $|1-\exp(x+y-u-v)| \leq 3|x+y-u-v|$ by the mean value theorem. Therefore we have
\begin{align*}
	\left|W(x,y) - W(u,v)\right| 
	&= \frac{\left|e^{u + v} - e^{x + y}\right|}{(1 + e^{x + y})(1 + e^{u +v})} 
	\le  \min \left\{ \frac{|e^{u+v-x-y}-1|}{1+e^{u+v}}, \frac{|1-e^{x+y-u-v}|}{1+e^{u+v}}\right\}\\
	&\leq 3|u+v-x-y| \min\{W(u,v), 1-W(u,v)\}, 
\end{align*}
By symmetry we obtain a similar upper bound with $W(x,y)$ instead of $W(u,v)$ and hence we conclude
\begin{equation}\label{eq:Wdelta}
	\left|W(x,y) - W(u,v)\right| \leq 3|u+v-x-y| \min\{W(x,y),1-W(x,y),W(u,v), 1-W(u,v)\}.
\end{equation}

Then, for any $x+y\leq u+v$ there are $c,d$ such that $x+y \leq c+d \leq u+v$, so by \eqref{eq:H'minbound} and \eqref{eq:Wdelta} we get
\begin{align*}
	|H(W(x,y)) - H(W(u,v))| &= |H'(W(c,d))(W(x,y)-W(u,v))| \\
	&\le  \frac{\min\{H(W(x,y)),H(W(u,v))\} \, |W(x,y)-W(u,v)|}{\min\{W(x,y),1-W(x,y),W(u,v), 1-W(u,v)\}}  \\
	&\le 3|x+y-u-v| \cdot \min\{H(W(x,y)),H(W(u,v))\}. \numberthis \label{eq:diff_H_bound}
\end{align*}

Next, for the partition $\rho_t$ we have
\begin{align*}
	|I_t| = \frac{2 R_n}{m_n - 1} \le \frac{2 R_n}{\log(n)^2} = \frac{\log(n) - \log(\nu \beta^2)}{\log(n)^2} \le \frac{1}{\log(n)},
\end{align*} 
for $t \ge 2$. In addition, \eqref{eq:partition_bounds_W_tilde_W} implies that for $(x, y) \in I_t \times I_s$, with $t, s \ge 2$,
\begin{align*}
	W(\rho_{t-1}, \rho_{s - 1}) \le \widetilde{W}_n(x,y), \, W(x,y) \le W(\rho_t, \rho_s),
\end{align*}
and thus there is a pair $x_{s,t}, y_{s,t} \in I_s \times I_t$ such that $\widetilde{W}_n(x,y) = W(x_{s,t},y_{s,t})$ on $I_s \times I_t$. 
Therefore, using \eqref{eq:diff_H_bound}, we get
\begin{align*}
	|H(W(x,y)) - H(\widetilde{W}_n(x,y))| &= |H(W(x,y)) - H(W(x_{s,t},y_{s,t}))| \\
	&\le 3|x+y-x_{s,t}-y_{s,t}| H(W(x,y))  \\ 
	&\le 3(|I_t| + |I_s|) H(W(x,y)) \le \frac{6}{\log(n)} H(W(x,y))
\end{align*}

Finally, integrating the above equation over the whole square $[-R_n,R_n] \times [-R_n, R_n]$ we obtain
\begin{align*}
	\left|\iint_{-R_n}^{R_n} \left(H\left(W(x,y)\right) - H\left(\widetilde{W}(x,y)\right)\right) \, d\mu_n(y) \, d\mu_n(x)\right| \leq \frac{6}{\log(n)} \sigma[W,\mu_n],
\end{align*}
which proves \eqref{eq:def_convergence_scaled_Gibbs_graphon_entropy} since
\begin{align*} 
	\lim_{n \to \infty} \frac{n}{\log(n)}|\sigma[W, \mu_n] - \sigma[\widetilde{W}_n, \mu_n]| 
	\leq \lim_{n\to \infty} \frac{6 n}{(\log n)^2}\sigma[W,\mu_n] = 0
\end{align*}
thanks to Theorem~\ref{thm:convergence_graphon_entropy}.

\paragraph{Acknowledgments}
This work was supported by the ARO grant No.\ W911NF-16-1-0391 and by the NSF grant No.\ CNS-1442999.  

\bibliographystyle{plain}

\end{document}